\documentclass[a4paper,11pt, oneside]{amsart}

\usepackage[cp1250]{inputenc}
\usepackage{amsmath}
\usepackage{amssymb,amsbsy,amsmath,amsfonts,amssymb,amscd}
\usepackage{latexsym}
\usepackage{amsthm}
\usepackage{mathrsfs}
\usepackage{graphics}
\usepackage{color}
\usepackage{tikz}
\usetikzlibrary{arrows}
\usepackage[new]{old-arrows}
\usepackage{tikz-cd}
\usepackage{pdfpages} %cover
\usepackage{longtable}
\usepackage{xy}
\usepackage{chngcntr}
\usepackage{array}
\usepackage{color}
\usepackage{comment}
\usepackage[shortlabels]{enumitem}
\usepackage{geometry}
\usepackage{setspace}
\usepackage{multirow}
\usepackage[pagebackref=true]{hyperref}
\usepackage{textcmds} % for english "
\usepackage{pgf, tikz}
\usetikzlibrary{3d,angles,quotes,calc,shapes,decorations,positioning,intersections,through,angles, patterns, babel, trees, mindmap, tikzmark} %for geometric constructions
\usepackage{subcaption}
\usepackage{caption}
\usepackage[alphabetic,backrefs,lite]{amsrefs} % for bibliography
\usepackage{amsthm}
\input xy
\xyoption{all}

\DeclareMathOperator{\Stab}{Stab}

\DeclareMathOperator{\Aut}{Aut}

\DeclareMathOperator{\ord}{ord}

\DeclareMathOperator{\GL}{GL}

\DeclareMathOperator{\Deck}{Deck}

\DeclareMathOperator{\PSL}{PSL}

\def\eea{\end{eqnarray*}}
\def\bea{\begin{eqnarray*}}

\newcommand\dual{\mathrel{\raise3pt\hbox{$\underline{\mathrm{\thinspace d
\thinspace}}$}}}
\newcommand\qe{\ifhmode\unskip\nobreak\fi\quad $\Box$}       % box for QED

\def\BOX{\hfill\lower.5\baselineskip\hbox{$\Box$}}
\DeclareMathOperator{\SheafHom}{\mathscr{H}\text{\kern -4pt {\textit{om}}}\,}
\DeclareMathOperator{\SheafExt}{\mathscr{E}\text{\kern -3pt {\textit{xt}}}\,}

\newtheorem{theorem}{Theorem}

\newtheorem{Theorem}[theorem]{Theorem}

\newtheorem{Proposition}[theorem]{Proposition}

\newtheorem{Corollary}[theorem]{Corollary}

\newtheorem{Lemma}[theorem]{Lemma}
\newtheorem{example}[theorem]{Example}

\numberwithin{theorem}{section}
\numberwithin{equation}{section}

\theoremstyle{definition}
\newtheorem{Remark}[theorem]{Remark}
\newtheorem{remark}[theorem]{Remark}
\newenvironment{rem}{\begin{remark}\rm}{\end{remark}}
\newtheorem{defin}[theorem]{Definition}

\newenvironment{definition}{\begin{defin}\rm}{\end{defin}}

\makeatletter
\def\tagform@#1{\maketag@@@{\ignorespaces#1\unskip\@@italiccorr}}
\makeatother

\newcolumntype{H}{@{}>{\lrbox0}l<{\endlrbox}} %Spalten ignorieren
\newcommand{\mylabel}[2]{#2\def\@currentlabel{#2}\label{#1}}

\hypersetup{colorlinks=true, unicode=true, linkcolor=[rgb]{0.10,0.05,0.67}, citecolor=[rgb]{0.10,0.05,0.67}, filecolor=[rgb]{0.10,0.05,0.67}, urlcolor=[rgb]{0.10,0.05,0.67}}

\begin{document}

\title{On Rigid Varieties Isogenous to a Product of Curves}
\author{ Federico Fallucca, Christian Gleissner and Noah Ruhland}
%\address{Federico Fallucca and Christian Gleissner  \newline University of Bayreuth, Universit\"atsstr. 30, D-95447 Bayreuth, Germany }
%\email{christian.gleissner@uni-bayreuth.de,Â  federico.fallucca@uni-bayreuth.de}

\thanks{
\textit{Acknowledgements: The authors would like to express their gratitude to Ingrid Bauer, Fabrizio Catanese and Roberto Pignatelli for valuable discussions about this project. The first author held a research grant from INdAM, Istituto Nazionale di Alta Matematica.}}
\address{Federico Fallucca 
\newline University of Milano-Bicocca , Via Roberto Cozzi 55, 20126 Milano, Italy}
\email{federico.fallucca@unimib.it}
\address{Christian Gleissner 
\newline University of Bayreuth, Universit\"atsstr. 30, D-95447 Bayreuth, Germany}
\email{Christian.Gleissner@uni-bayreuth.de}
\address{Noah Ruhland 
\newline University of Bayreuth, Universit\"atsstr. 30, D-95447 Bayreuth, Germany}
\email{Noah.Ruhland@uni-bayreuth.de}

\begin{abstract}

In this note, we study rigid complex manifolds  that are realized as quotients of a product of curves by a free action of a finite group. They serve as higher-dimensional analogues of Beauville surfaces. 
Using uniformization, we outline the theory to characterize these manifolds through specific combinatorial data associated with the group under the assumption that the action is diagonal and the manifold is of general type. This leads to the notion  of a $n$-fold Beauville structure.
We define an action on the set of all $n$-fold Beauville structures of a  given finite group that allows us to distinguish the biholomorphism classes of the underlying rigid manifolds.  As an application, we give a classification of these manifolds  with group $\mathbb Z_5^2$ in the three dimensional case and prove that this is the smallest possible group that allows a rigid, free and diagonal  action on a product of three curves. In addition, we provide the classification of rigid 3-folds $X$ given by a group acting   faithfully on each factor for any value of the holomorphic Euler number   $\chi(\mathcal O_X) \geq  -5$. 
\end{abstract}

\maketitle

%\tableofcontents
\section{Introduction}

 A central aspect of any kind of geometric consideration is to study the symmetries of the underlying spaces. In complex algebraic geometry, the symmetries are  biholomorphic  self-maps of projective manifolds. Given a known manifold $X$ then, in the spirit of Godeaux,  we can often construct new and interesting manifolds by taking  quotients $X/G$ 
 modulo a free action of a  finite group of automorphisms. 
 An important class of such quotients are the varieties isogenous to a product. 
 \begin{definition}\label{VarIso}
 A complex variety $X$ is isogenous to a product if it is isomorphic to a quotient 
 \[
X\simeq  (C_1 \times \ldots \times C_n)/G,
 \]
 where the $C_i's$ are compact Riemann surfaces of genus $g(C_i)\geq 1$ and $G$ is a finite group acting freely on the product $C_1\times \ldots \times C_n$. 
 We call $X$ \emph{isogenous to a higher product}, if $g(C_i) \geq 2$ for all $C_i$. 
 \end{definition} Since Catanese introduced these varieties  in 
 \cite{Cat00}, they have been  studied extensively, especially in dimension two in order to construct 
 and classify surfaces of general type and describe their moduli spaces, cf. \cite{BCG08, CP09, G15, P11, fede24}.
We point out that a  variety $X$ isogenous to a product is of general type if and only if it is isogenous to a higher product. It has been shown by Catanese that a 
 surface $S$ isogenous to a higher product has a unique minimal realization $$S\simeq (C_1\times C_2)/G.$$ It is characterized by the property that the diagonal subgroup $G_0 :=G\cap (\Aut(C_1) \times \Aut(C_2))$ acts faithfully on each curve $C_i$. \\
An easy example of a surface isogenous to a higher product  is due to Beauville \cite[Exercises X.13 (4) p.118]{B83}: consider the product of two Fermat quintics $$C=\lbrace x_0^5+x_1^5+x_2^5=0\rbrace \subset \mathbb P_{\mathbb C}^2$$ together with the action of $\mathbb Z_5^2$ defined by 
  \begin{equation}\label{BeauvillesBeauville}
  (a,b) \ast \big([x_0:x_1:x_2],[y_0:y_1:y_2]):=([\zeta_5^ax_0:\zeta_5^bx_1: x_2],[\zeta_5^{a+3b}y_0:\zeta_5^{2a+4b}y_1: y_2]\big), \quad (a,b)\in \mathbb Z_5^2. 
  \end{equation}
  Since this action is free, the quotient $$S:=(C\times C)/\mathbb Z_5^2$$ is a surface isogenous to a product. Its holomorphic Euler number is 
  $\chi(\mathcal O_S)=1$, which  is the minimal Euler number for a surface of general type. Remarkably, it turns out that $S$ is rigid, i.e. it has no non-trivial deformations. Motivated by this example Catanese defined Beauville surfaces as rigid 
 surfaces isogenous to a product \cite[Def 3.23]{Cat00}.  They are always of general type (\cite[cf. Prop 3.2]{IFG}, \cite[Thm 2.7]{BC18}) and have an 
  entirely group theoretical description in terms of a so called Beauville structure of the corresponding  group. For this reason, Beauville surfaces have been  actively studied,  not only by  algebraic geometers, but also by  group theorists, as they provide a rich framework to explore the interplay between these  disciplines (cf. \cite{BGV15}).
  The aim of this paper is to extend the theory of Beauville surfaces to higher dimensions and explore some new phenomena. 
  For this purpose we define:
  \begin{definition}
      A rigid variety isogenous to a product is called a Beauville manifold. 
  \end{definition}
  In contrast to the surface case, the geometry of   higher dimensional  Beauville manifolds is more involved. First we point out that they  are not necessarily of general type. Indeed, according to \cite[Thm 3.4 and Thm 3.5]{BC18} we have:
  \begin{itemize}
  \item 
 For all $n\geq 4$ there exists a Beauville $n$-fold of Kodaira dimension $0$.
\item 
For all $n\geq 3$ there exists a Beauville  $n$-fold of Kodaira dimension $\kappa$ for all $2\leq \kappa \leq n$.
 \end{itemize}
Our second remark is that even in the case of a higher product, we cannot assume that the diagonal subgroup 
\[
G_0 :=G\cap (\Aut(C_1) \times \ldots \times \Aut(C_n))
\]
acts faithfully on each curve $C_i$. Indeed, in \cite{FG16} the authors provide a classification of all 3-folds $X$ isogenous to a higher product of curves with  $\chi(\mathcal O_X)=-1$ under the assumption that 
the action is diagonal, i.e.  $G=G_0$ and faithful on each curve $C_i$. Among the $54$  families there are  no  rigid examples. However, dropping the faithfulness on the factors, a rigid  example with 
$\chi(\mathcal O_X)=-1$ is easy to construct as a modification of the original Beauville surface \ref{BeauvillesBeauville}: 
we take the hyperelliptic curve $D$  of genus two given by the affine equation 
$y^2=x^5-1$ and define $X$  as the quotient of $C^2\times D$ modulo the following free $\mathbb Z_5^2$-action: 
  \[
  (a,b) \ast \big([x_0:x_1:x_2],[y_0:y_1:y_2],[x,y]):=([\zeta_5^ax_0:\zeta_5^bx_1: x_2],[\zeta_5^{a+3b}y_0:\zeta_5^{2a+4b}y_1: y_2],[\zeta_5^a x,y]\big). 
  \]
Clearly, the action is not faithful on the third factor. However, the quotient $X$ is rigid and 
$\chi(\mathcal O_X)=-1$. 
This observation  serves as a motivation for the following questions about Beauville 3-folds of general type:
\begin{itemize}
    \item 
    Is $\mathbb Z_5^2$ the smallest group attached to a  Beauville $3$-fold?
    \item What is the number of  biholomorphism classes of Beauville 3-folds with group $\mathbb Z_5^2$ and Euler number $\chi(\mathcal O_X)=-1$?
    \item What is the largest integer $\chi \leq -2$ such that there exists a Beauville threefold $X$ with $\chi(\mathcal O_X)=\chi$ under the assumption that the  group acts faithfully  on each $C_i$? Is it possible to  classify these 3-folds up to biholomorphism?
\end{itemize}
For simplicity, we restrict our focus to unmixed Beauville $3$-folds, i.e. to the case where $G$ acts diagonally on the (higher) product of curves. Our results are the following: 

\begin{Theorem}\label{THMI}
The smallest group attached to an unmixed Beauville 3-fold $X$ is $\mathbb Z_5^2$. These 3-folds have either $\chi(\mathcal{O}_X)=-1$ or $-5$. There are $8$ biholomorphism classes of such 3-folds  with $\chi(\mathcal{O}_X)=-1$ and $76$ with $\chi(\mathcal{O}_X)=-5$.
\end{Theorem}

\begin{Theorem}\label{THMII}
The number $\mathcal N$ of biholomorphism classes of unmixed Beauville $3$-folds $X$ with holomorphic Euler number $\chi(\mathcal O_X)=\chi \in \lbrace -5, -4, -3, -2  \rbrace$, such that the corresponding group acts faithfully on each factor of the product is:  
\begin{table}[!ht]
\begin{tabular}{c|cccc}
$\chi$  & -5 & -4& -3& -2 \\
\hline
\hline
$\mathcal N$ & 77 & 8 & 0 & 1  \\
\end{tabular}
\end{table}
\end{Theorem}
We will now explain how the paper is organized. In Section~\ref{generalities}, we 
recall the basic theory of varieties isogenous to a product of curves and introduce Beauville manifolds. Crucial for our analysis  is the existence and uniqueness of a minimal realization of a Beauville manifold or more generally of a variety isogenous to a higher product of curves of unmixed type. Even though this result is well-known in the surface case \cite[Cor. 3.12 and Prop. 3.13] {Cat00} and folklore in higher dimensions, we decided to include a proof, since we could not find a reference for dimension $n\geq 3$. 
 In Section \ref{grouptheo}, we use Riemann's existence theorem to give a purely group theoretical description of unmixed Beauville manifolds of general type. More precisely, we explain how to attach to the minimal realization of a given Beauville $n$-fold an \emph{$n$-fold 
Beauville structure} of the corresponding group $G$.
Then we provide a natural action of the group $\Aut(G)\times (\mathcal B_3 \wr \mathfrak S_n)$  on the set $\mathcal {UB}_n(G)$ of all (unmixed) Beauville structures of $G$, where $\mathcal B_3$ is the Artin-Braid group on three strands. 
We show that the biholomorphism classes of 
unmixed Beauville $n$-folds are in $1:1$ correspondence with the orbits of this action.
Unfortunately, for certain  groups, it can be very difficult and computationally expensive to determine these orbits.
%even in case of  $G=\mathbb Z_5^2$ and $n=3$. 
This difficulty is resolved in  
Section \ref{sec: computation_o_the_fibre}, where we present an effective method to compute the orbits, that  extends the \emph{automorphism of Braid type} approach of  \cite[Section 1.2]{fede24} from surfaces to arbitrary dimensions. These results allow us to employ the \emph{Database of topological types of actions on curves} of \cite{CGP23} for explicit computations.
In Section \ref{section: BeauvilleDim}, we discuss the notion of the Beauville dimension of a finite group $G$, which is the minimum dimension of a  Beauville manifold with group $G$. 
This concept was introduced by Carta and Fairbairn in \cite{CF22} under the assumption that the  group acts faithfully on each factor. 
We drop this assumption and show that the group $\mathbb Z_n^3$ has Beauville dimension three if and only if $\gcd(n,6)=1$. 
In Section \ref{section: Examples}, we prove our main theorems using a  MAGMA \cite{BCP97} implementation of the algorithm from Section  \ref{sec: computation_o_the_fibre}. 
 The reader can find a MAGMA implementation on the webpage 
\begin{center}
\url{https://www.komplexe-analysis.uni-bayreuth.de/de/team/gleissner-christian/index.php}.
\end{center}

\section{Generalities on Beauville manifolds}\label{generalities}

In this section, we present the basic theory of varieties isogenous to a product of curves via uniformization. More precisely, we extend the results \cite[Cor. 3.9, Prop. 3.11, Cor. 3.12, Prop. 3.13]{Cat00} from Catanese on surfaces, which were established using different methods. 

%In this section, we recall the basic theory of varieties isogenous to a product of curves,  which was developed for surfaces in \cite[Cor. 3.9, Prop. 3.11, Cor. 3.12, Prop. 3.13]{Cat00}. %\textcolor{red}{Should we cite him better below?}
In particular, we focus on Beauville manifolds, which are the rigid varieties isogenous to a product of curves. In higher dimensions they serve as natural generalizations of Beauville surfaces.

%\begin{definition}\label{GiG0} \
%\begin{enumerate}
%\item 
%A complex $n$-dimensional algebraic variety $X$ is isogenous to a  product of curves if there exists compact Riemann surfaces $C_1, \ldots,C_n$ with genus $g(C_i) \geq 1$ and a finite group $G\leq \Aut(C_1\times \ldots \times C_n)$ acting freely such that 
%\[
%X\simeq (C_1 \times \ldots \times C_n)/G. 
%\]
%\item 
%We say that $X$ is \textit{isogenous to a higher product} if it is isogenous to a product of curves with  genus $g(C_i)\geq 2$.  
%\end{enumerate}
%\end{definition}

Recall that a variety $X$ isogenous to a product is smooth and  projective. The Kodaira dimension $\kappa(X)$ is equal to the number of curves of genus $g(C_i) \geq 2$ in the product $C_1\times \ldots \times C_n$. 
The $n$-fold self-intersection of the canonical class $K_X^n$, the topological Euler number $e(X)$ and the holomorphic Euler number $\chi(\mathcal O_X)$ are given 
 in terms of the genera $g(C_i)$ and the order of the  group. 

\begin{Proposition}\label{Global_Invar}
Let $X\simeq (C_1 \times \ldots \times C_n)/G$ be a variety isogenous to a product, then 
\[
\chi(\mathcal O_X)=\frac{(-1)^n}{\vert G\vert} \prod_{i=1}^{n}\big(g(C_i)-1\big), \quad K_X^n=(-1)^n n! 2^n\chi(\mathcal O_X) \quad \makebox{and} \quad e(X)=2^n\chi(\mathcal O_X). 
\]
\end{Proposition}

\begin{definition}
    A rigid variety isogenous to a product is called a Beauville manifold. 
\end{definition}

\begin{remark}
    A Beauville surface is always isogenous to a higher product, i.e. of general type (\cite[cf. Prop. 3.2]{IFG} and \cite[Thm. 2.7]{BC18}).  This is not true anymore for Beauville manifolds of higher dimension:
    \begin{enumerate}
\item 
For all $n\geq 4$ there exists a Beauville $n$-fold of Kodaira dimension $0$ (see \cite[Thm. 3.4]{BC18}).
\item 
For all $n\geq 3$ there exists a Beauville  $n$-fold of Kodaira dimension $\kappa$ for all $2\leq \kappa \leq n$ (see \cite[Thm. 3.5]{BC18}).
 \item There is no rigid and free action on a 3-dimensional complex torus. In particular, there are no
 Beauville $3$-folds of Kodaira dimension $0$. (see \cite[Thm. 1.1(a)]{DG23}).
 \item 
 The existence of Beauville manifolds of Kodaira dimension 1 is still an open question. However, 
 there are no  such manifolds if the action is   diagonal and  faithful  on each factor, cf. Definition \ref{defn: absolutely_faithful}. For this result, see \cite[Cor. 3.11]{BGK25}. 
 \item 
  The existence of rigid 3-folds of Kodaira dimension 0 and 
 rigid $n$-folds of Kodaira dimension 1 for all $n\geq 3$ is known (see \cite{B82} and \cite{BG20}). They are obtained as certain resolutions of quotients of a product of curves by a  non-free action.  
  \end{enumerate}
\end{remark}
In this paper, we are mainly interested in the special case where $X$ is isogenous to a higher product of curves. In order to study group actions on a product $C_1 \times \ldots \times  C_n$ of compact Riemann surfaces with $g(C_i) \geq  2$, it is important to understand the structure of the automorphism group of the product. This group has a simple description in terms of the automorphism groups $\Aut(C_i)$ of the factors, thanks to the lemma below: %\textcolor{red}{(cf. \cite[Corollary 3.9]{Cat00})}:

\begin{Lemma}\label{rigidity-lemma}
Let $D_1, \ldots ,D_k$ be pairwise non-isomorphic compact Riemann surfaces with $g(D_i) \geq 2$. 
Then for all positive integers $n_1, \ldots ,n_k$ it holds:  
\[
\Aut(D_1^{n_1} \times \ldots \times D_k^{n_k})= (\Aut(D_1) \wr 
\mathfrak S_{n_1}) \times \ldots \times 
(\Aut(D_k) \wr  \mathfrak S_{n_k}). 
\]
Here $\Aut(D_i) \wr 
\mathfrak S_{n_i}=\Aut(D_i)^{n_i} \rtimes
\mathfrak S_{n_i}$ denotes the \emph{wreath product}. 
\end{Lemma}

\begin{proof}
    Any automorphism $\varphi \in \Aut(D_1^{n_1} \times \ldots \times D_k^{n_k})$ lifts to an automorphism 
    $\hat{\varphi}$ of the universal cover, which is a product of unit discs $\Delta^n$. 
    The claim follows from the well known fact that 
    \[
    \Aut(\Delta^n) = \Aut(\Delta) \wr \mathfrak S_n. 
    \]
    See,  
    \cite[Proposition 3, p.68]{N71}.  
\end{proof}

The above lemma motivates the following definition: 

\begin{definition}
A  $n$-dimensional variety $X$ isogenous to a higher product is said to be of unmixed type, if there is a realization 
$X\simeq (C_1 \times \ldots \times C_n)/G$, 
such that $G\leq \Aut(C_1)\times \ldots \times \Aut(C_n)$. Otherwise, we say that $X$ is of mixed type. 
\end{definition}

\begin{Proposition}%(cf. \cite[\textcolor{red}{Proposition 3.11}]{Cat00})
    A variety $X$ is isogenous to a higher product of curves if and only if there exists an unramified cover 
    \[
   f\colon  C_1 \times \ldots \times C_n \to X, \qquad \makebox{where} \quad g(C_i) \geq 2. 
    \]
\end{Proposition}

\begin{proof}
    Assume that $f$ exists, then the universal cover of $X$ is the polydisc $\Delta^n$ and the cover $f$ is induced by a finite index subgroup 
    \[
    \Gamma_1 \times \ldots \times \Gamma_n < \Gamma, 
    \]
  where $\Gamma < \Aut(\Delta^n)$  is the  fundamental group of $X$ and $\Gamma_i$ the fundamental group of $C_i$. An element of  $ \Gamma$ is of the form 
  \[
  \gamma_{\tau}(z_1,\ldots, z_n)=(\gamma_1(z_{\tau(1)}), \ldots, \gamma_n(z_{\tau(n)})), 
  \]
  for some permutation $\tau\in \mathfrak S_n$ and automorphisms $\gamma_i \in \Aut(\Delta)$. 
  This implies 
  \[
   \gamma_{\tau} (\Gamma_1\times \cdots \times \Gamma_n)\gamma_{\tau}^{-1}=
   \gamma_1\Gamma_{\tau(1)} \gamma_1^{-1} \times \cdots \times \gamma_n\Gamma_{\tau(n)} \gamma_n^{-1}. 
  \]
  The normal core i.e. the largest subgroup of $\Gamma_1\times \cdots \times \Gamma_n$ that is normal in $\Gamma$ is therefore given by 
\[
{\rm core}_\Gamma(\Gamma_1\times \cdots \times \Gamma_n)= \bigcap_{\gamma_{\tau} \in \Gamma } \gamma_{\tau} (\Gamma_1\times \cdots \times \Gamma_n)\gamma_{\tau}^{-1}= \prod_{i=1}^n\bigcap_{\gamma_\tau \in \Gamma} \gamma_i\Gamma_{\tau(i)}\gamma_i^{-1}.
\]
Since the core of a finite index subgroup has finite index, we obtain a finite Galois cover 
\[
\Delta/\Gamma_1' \times \cdots \times \Delta/\Gamma_n' \to X, \quad \makebox{where} \quad \Gamma_i':= \bigcap_{\gamma_\tau \in \Gamma} \gamma_i\Gamma_{\tau(i)}\gamma_i^{-1}.
\]
\end{proof}

We want to point out that a realization of a variety $X$  isogenous to a higher product as a quotient 
$(C_1 \times \ldots \times C_n)/G$
is in general  not unique. However, we can 
always find a so called  \emph{minimal realization} which is unique up to isomorphism. 
In this paper, we want to stick to the case where 
$X$ is of unmixed type. Then, given a realization, we obtain $G$-actions 
$\psi_i \colon G \to \Aut(C_i)$
on the factors which are not 
 necessarily faithful. We denote by  $K_i$ the kernel of $\psi_i$ and define 
the quotient $G/K_i$, which then acts faithfully on $C_i$ as $\overline{G}_i$.
%The following definition is crucial:

\begin{definition}\label{defn: absolutely_faithful}
    A diagonal $G$-action on $C_1\times \ldots \times C_n$ is called 
    \begin{enumerate}
    \item minimal, if $K_1 \cap \ldots \cap \widehat{K}_i \cap \ldots \cap K_n=\lbrace 1_G \rbrace$ for all $i$. 
\item absolutely faithful, if all kernels $K_i$  are trivial. 
\end{enumerate}
\end{definition}
Clearly, an absolutely faithful action is also minimal and in dimension two the notions coincide.

\begin{Theorem}\label{minimalreal}
Every variety isogenous to a higher product of curves of unmixed type  has a unique minimal realization, i.e. a realization obtained by a minimal action. 
\end{Theorem}

To prepare for the proof of this theorem, we need to recall the structure of the fundamental group of a variety isogenous to a higher product of unmixed type, cf. \cite{DP12}.

\begin{remark}\label{structureFund}
   Let $(C_1\times \ldots \times C_n)/G$ be a not necessarily minimal realization 
of a variety $X$ isogenous to a higher product of curves of unmixed type.  Considering the universal cover $\pi_i \colon \Delta \to C_i$, we obtain the short exact sequence 
        \[ 1 \to \Gamma_i \to \mathbb T_i \stackrel{\rho_i}{\to} \overline{G}_i \to 1. \]
    As above $\Gamma_i$ is isomorphic to  the fundamental group of $C_i$ and   $\mathbb T_i$ is the group of all possible  lifts of the elements in $\overline{G}_i$, i.e. the orbifold fundamental group: 
        \[ \mathbb T_i = \{ \gamma \in \Aut(\Delta) ~ | ~ \textrm{\ exists \ }\overline{g}\in \overline{G}_i  \textrm{ \  such that \ } \pi_i\circ \gamma= \psi_i(g)\circ \pi_i\}, \]
            see \cite[Chapter 6]{Cat15} for an in depth discussion of orbifold fundamental groups. 
   Similarly, we  take the universal cover $\Delta^n \to C_1 \times \ldots \times C_n$ and get the short exact sequence
        \[ 1 \to \Gamma_1 \times \ldots \times \Gamma_n \to \Gamma \to G \to 1. \]
   The group  $\Gamma$ consists  of the lifts of the  elements in $G$ to  $\Aut(\Delta^n)$. It is 
     isomorphic to the fundamental group of $X$, because the action of $G$ is free. Since $G$ acts diagonally on the product of curves, an automorphism $g \in G$ lifts to an automorphism $(\gamma_1 , \ldots, \gamma_n) \in \Aut(\Delta)^n$, if and only if each $\psi_i(g)$ lifts to $\gamma_i \in \Aut(\Delta)$. Therefore, we can write $\Gamma$ in the following way:
        \[ \Gamma = \{ (\gamma_1,\ldots,\gamma_n) \in \mathbb T_1 \times \ldots \times \mathbb T_n ~ | ~ \overline{g} = \rho_i(\gamma_i) \in \overline{G}_i \text{ for all } 1 \leq i \leq n \text{ and some } g \in G \}.  \]
\end{remark}

\begin{proof}[Proof of Theorem \ref{minimalreal}]
(I) To show the  existence of a minimal realization, we start with an  arbitrary  realization 
$$X\simeq \frac{C_1\times \ldots\times C_n}{G}.$$ For  each $i$, we consider  the normal subgroups 
$H_i:=K_1 \cap \ldots \cap \widehat{K}_i \cap \ldots \cap K_n
\trianglelefteq G$. 
Note that $H_i$ acts trivially on $C_j$ for all  $i\neq j$ and freely on $C_i$.  In particular, the genus of the quotient curve
$C_i/H_i$ is at least $2$. 
We take the product  $H:=H_1\cdot \ldots \cdot H_n \trianglelefteq G$ of our normal subgroups  and form 
the double  quotient 
\[
X\simeq \frac{(C_1\times \ldots\times C_n)/H}{G/H} \simeq 
\frac{C_1/H_1 \times \ldots \times C_n/H_n}{G/H}. 
\]
By construction, the induced $G/H$-action on the product $C_1/H_1\times \ldots \times C_n/H_n$ is minimal. \\
(II) To prove the uniqueness, we consider a biholomorphism between two minimal realizations of a variety $X$ isogenous to a higher  product: 
\[
f\colon \frac{C_1\times \ldots\times C_n}{G} \to \frac{D_1\times \ldots\times D_n}{G'}. 
\]
The map $f$ lifts to an automorphism 
$\hat{f} \in \Aut(\Delta^n) = \Aut(\Delta)\wr \mathfrak S_n$. Up to permutation of the curves $D_i$, we may assume that 
$$\hat{f}(z_1,\ldots,z_n)=\big(\hat{f_1}(z_1), \ldots,\hat{f_n}(z_n)\big).$$
Conjugation with $\hat{f}$ induces an isomorphism between 
the deck transformation  groups of the universal covers
\[
f_{\ast} \colon \Gamma \to \Gamma', 
\qquad (\gamma_1, \ldots, \gamma_n) \mapsto 
\big(\hat{f_1}\gamma_1\hat{f_1}^{-1}, \ldots, \hat{f_n}\gamma_n\hat{f_n}^{-1}\big).
\]
Our goal is to  show that $f_{\ast}$ restricts to an isomorphism between the fundamental groups of the product of curves 
\begin{equation}\label{FundProducts}
f_{\ast}(\Gamma_1\times \ldots \times \Gamma_n)= \Gamma_1'\times \ldots \times \Gamma_n'.
\end{equation}
Under this condition the uniqueness follows, since then  $f_{\ast}$ induces an isomorphism of Galois groups 
$$G\simeq \Gamma/(\Gamma_1\times \ldots \times \Gamma_n)
\to \Gamma'/(\Gamma_1'\times \ldots \times \Gamma_n') \simeq G'$$
and  a biholomorphic  lift of $f$ to the 
 product of curves:
	 \[
	\xymatrix{
	(C_1\times \ldots\times C_n)/G  \ar[r]^{f} & 
 (D_1 \times \ldots \times D_n)/G'  \\
	C_1\times \ldots \times C_n \ar[u] \ar[r]^{\tilde{f}} & D_1\times \ldots \times D_n. \ar[u]}
	\]
To verify equation \ref{FundProducts}, it suffices to show that 
\[
f_{\ast}(1,\ldots,1, \gamma_i ,1,\ldots,1) \in \Gamma_1'\times \ldots \times \Gamma_n', 
\]
for all $1\leq i \leq n$ and $\gamma_i\in \Gamma_i$. Here we use the structure of $\Gamma'$ given in Remark \ref{structureFund} and assume for simplicity  that $i=1$. 
Take $\gamma_1 \in \Gamma_1$, then 
\[
f_{\ast}(\gamma_1,1,\ldots,1)=
\big(\hat{f_1}\gamma_1\hat{f_1}^{-1},1, \ldots, 1\big)
\in \Gamma'. 
\]
Hence there exists $g\in G'$  such that 
\[
\overline{g}=\rho_1'\big(\hat{f_1}\gamma_1\hat{f_1}^{-1}\big)\in G'/K_1 \qquad \makebox{and} \qquad  \overline{g}=\rho_j'(1)=1 \in G'/K_j'
\]
for all $j\geq 2$. 
This implies 
$g\in K_2'\cap \ldots \cap K_n' =\lbrace 1_{G'}\rbrace$
and therefore also $\hat{f_1}\gamma_1\hat{f_1}^{-1} \in \ker(\rho_1')=\Gamma_1'$. 
\end{proof}

\begin{remark}
 Theorem \ref{minimalreal} allows us to attach, to any   variety  isogenous to a higher product of unmixed type, 
a product of curves 
$C_1\times \ldots \times C_n$ together with a finite group $G$ acting minimally on the product. We will use this fact in the next section to give a purely group theoretical description of Beauville manifolds. 
\end{remark}

\begin{Proposition}
A diagonal and free $G$-action  on a 
product of compact Riemann surfaces 
\[
Y=C_1 \times \ldots \times C_n \qquad \makebox{with} \qquad g(C_i) \geq 2
\]
yields a Beauville manifold, if and only if each $C_i$ is a \emph{triangle curve}: 
\begin{enumerate}
\item 
$C_i/\overline{G}_i \simeq \mathbb P^1$ and 
\item 
$C_i \to C_i/\overline{G}_i$ is branched in three points, where $\overline{G}_i=G/K_i$. 
\end{enumerate}
\end{Proposition}

\begin{proof}
    Since $H^0(C_i,\Theta_{C_i})=0$, the K\"unneth formula yields
    \[
    H^1(Y,\Theta_{Y})=H^1(C_1,\Theta_{C_1}) \oplus \ldots \oplus H^1(C_n,\Theta_{C_n}).  \]
    Since the $G$-action is assumed to be diagonal, we have that 
    \[
    H^1(Y,\Theta_{Y})^G=0 \qquad \makebox{if and only if} \qquad H^1(C_i,\Theta_{C_i})^{\overline{G}_i}=0.
    \]
    By  using  \cite[Examples VI.12 (2)]{B83}, the condition 
    $H^1(C_i,\Theta_{C_i})^{\overline{G}_i}=0$ is easily seen to be equivalent to $(1)$ and $(2)$.  
\end{proof}

\begin{Corollary}
A Beauville manifold is always regular, i.e. it has no  non-zero global holomorphic $1$-forms. 
\end{Corollary}

\begin{proof} 
The irregularity of $X=(C_1\times \ldots \times C_n)/G$ is  given by 
$q(X)=\sum_{i=1}^n g(C_i/\overline{G}_i)=0$.
\end{proof}

\section{Group theoretical description of  Beauville $n$-folds}\label{grouptheo}
In this section we briefly recall  the theory of \emph{triangle curves} from the group theoretical point of view. 
This allows us to give a group theoretical description of unmixed Beauville $n$-folds that provides a way to classify them up to biholomorphism. \\
As we mentioned in the previous section, triangle curves are finite Galois covers
of the projective line branched on three points $\mathcal B:=\lbrace -1,0,1 \rbrace \subset \mathbb P^1$.
Note that the fundamental group of the complement $\mathbb{P}^1 \setminus \mathcal{B}$ is generated by three simple loops $\gamma, \delta$ and $\epsilon$ around  the points  $-1, 0$ and $1$, respectively. These loops satisfy a single relation and we get: 
\[
\pi_1(\mathbb{P}^1 \setminus \mathcal{B}, \infty) = \langle \gamma, \delta, \epsilon \mid \gamma \cdot \delta \cdot \epsilon = 1 \rangle.
\vspace{-1cm}
\]

\begin{center}
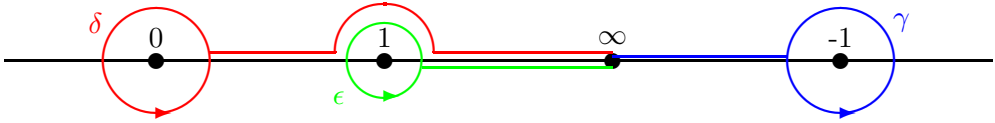
\begin{figure}[h]
\setlength{\unitlength}{1cm}
\thicklines
\begin{picture}(15,4)
    \put(1,2){\line(1,0){13}}
    \put(3,2){{\circle*{0.2}}}
    \put(6,2){{\circle*{0.2}}}
    \put(9,2){{\circle*{0.2}}}
    \put(12,2){{\circle*{0.2}}}
    \put(3,2){{\color{red}\circle{1.4}}}
    \put(12,2){{\color{blue}\circle{1.4}}}
    \put(6,2){{\color{green}\circle{1}}}
    \put(6,2.1){{\color{red}\oval(1.3,1.3)[t]}}
    \put(3.7,2.1){\color{red}\line(1,0){1.65}}
    \put(6.65,2.1){\color{red}\line(1,0){2.35}}
    \put(6.5,1.9){\color{green}\line(1,0){2.5}}
    \put(9,2.05){\color{blue}\line(1,0){2.3}}
    \put(3,2.3){\makebox(0,0){0}}
    \put(6,2.3){\makebox(0,0){1}}
    \put(9,2.3){\makebox(0,0){$\infty$}}
    \put(12,2.3){\makebox(0,0){-1}}
    \put(12.8,2.5){\color{blue}\makebox(0,0){$\gamma$}}
    \put(2.2,2.5){\color{red}\makebox(0,0){$\delta$}}
    \put(5.4,1.5){\color{green}\makebox(0,0){$\epsilon$}}
    \put(3.2,1.3){\color{red}\vector(1,0){0}}
    \put(6.2,1.53){\color{green}\vector(1,0){0}}
    \put(12.2,1.3){\color{blue}\vector(1,0){0}}
\end{picture}
\vspace{-1.5cm}

\caption{The three generators $\gamma$, $\delta$, $\epsilon$ of the fundamental group of $\mathbb{P}^1 \setminus \mathcal{B}$.}
\end{figure}
\end{center}
 \vspace{-1cm}
 
\begin{definition}
Let $G$ be a finite group. A triple $S=[a,b,c]$ of non-trivial group elements is called a {\em spherical triple of generators} or shortly a  \emph{generating triple} of $G$ if
\[
G= \langle a,b,c \rangle \qquad \makebox{and} \qquad  a\cdot b \cdot c=1_G
\]
The \emph{type of $S$} is  defined as $$T(S):=[\ord(a),\ord(b),\ord(c)].$$
\end{definition}
Observe that a  generating triple $S=[a,b,c]$ of a finite group $G$ induces a surjective homomorphism
\[
\eta_S \colon \pi_1(\mathbb P^1 \setminus \mathcal B, \infty ) \to G \qquad \makebox{by} \qquad  
    \gamma \mapsto a, \quad \delta \mapsto b, \quad \text{and} \quad \epsilon \mapsto c. 
    \]
Using \emph{Riemann's existence theorem}, the homomorphism $\eta_S$ yields a Galois triangle cover $$f_S \colon (C_S,q_0) \to (\mathbb P^1,\infty)$$ with branch locus $\mathcal B$ together with a unique isomorphism $\psi \colon G \to \Deck(f_S)$ such that the composition 
\[
(\psi \circ \eta_S) \colon \pi_1(\mathbb P^1 \setminus \mathcal B, \infty ) \to G \to  \Deck(f_S)
\]
is the monodromy map of the associated unramified cover. All triangle covers arise in this way.

\begin{rem}\label{triangle}\
\begin{enumerate}
    \item
    % A point $q\in C_S$ has non-trivial stabilizer if and only if it belongs to one of the fibres $f_S^{-1}(p_i)$. In this case, its stabilizer group has order $n_i=\ord(g_i)$.
The genus of $C_S$, the order of $G$ and the  orders of the generators $a$, $b$ and $c$ are related by  \emph{Hurwitz's formula}: 
$$
2g(C_S)-2=|G|\left(1- \frac{1}{\ord(a)}-  \frac{1}{\ord(b)}- \frac{1}{\ord(c)} \right).
$$
In particular we observe that $g(C_S) \geq 2$, if and only if 
\[
\frac{1}{\ord(a)}+  \frac{1}{\ord(b)}+ \frac{1}{\ord(c)}<1.
\]
In this case the generating triple $S=[a,b,c]$ is said to be  \emph{hyperbolic}, which we will assume throughout this section. 
%For this reason, $T(S)$ is also called the \emph{branching signature} of  $f_S$.
\item 
The stabilizer set of $S$ is defined as 
\[
\Sigma_{S}:=\bigcup_{g\in G} (g\langle a \rangle g^{-1} \cup g\langle b \rangle g^{-1} \cup g\langle c \rangle g^{-1} )
\]
It consists of the elements in 
$G$ which have  at least one fixed point on the curve $C_S$. 
\end{enumerate}
\end{rem}

\begin{rem}\label{ReCond}  
An unmixed  Beauville $n$-fold $X=(C_1\times\ldots\times C_n)/G$ of general type yields an $n$-tuple $[S_1,\ldots,S_n]$, such that 
\begin{enumerate}
\item 
$S_i=[a_i,b_i,c_i]$ is a hyperbolic generating triple of $\overline{G}_i=G/K_i$, where $K_i \trianglelefteq G$ is the kernel of $\psi_i$.
\item
The intersection 
$K_1 \cap \ldots \cap \widehat{K}_i \cap \ldots \cap K_n$ is trivial  for all $i$. 

\medskip
\item 
$\bigcap_{i=1}^n \Sigma_{S_i} \cdot K_i =\lbrace 1_G\rbrace $. 
\end{enumerate}

The  condition that $S_i$ is hyperbolic  tells us $g(C_{S_i})\geq 2$. 
The second condition reflects the \emph{minimality of the realization} and the third the \emph{freeness of the action}.  
Conversely, any such tuple gives rise to an unmixed  Beauville $n$-fold of general type.
\end{rem}

\begin{definition}
Let  $G$ be a finite group. An $n$-tuple $[S_1, \ldots, S_n]$ of hyperbolic generating triples such that the conditions (1), (2) and (3) from Remark \ref{ReCond} hold is called an \emph{$n$-fold unmixed Beauville structure} for $G$. The set of $n$-fold unmixed Beauville structures for $G$ is denoted by $\mathcal{UB}_n(G)$.  

The group $G$ is called an 
\emph{$n$-fold unmixed Beauville group} if $G$ admits an \emph{$n$-fold unmixed Beauville structure}. 
\end{definition}

In the following  we derive a 
criterion   which allows us to decide  whether  two  unmixed Beauville structures  yield biholomorphic Beauville manifolds  (cf. \cite[Proposition 4.2]{IFG} for the surface case).
For this purpose it is important to understand when two generating triples lead to the same triangle cover.

\begin{definition}\label{twistedCoverIso}
A \emph{twisted covering isomorphism} of two triangle $G$-covers 
$f_i\colon C_i \to \mathbb P^1$, 
branched on 
$\mathcal B=\lbrace -1,0,1 \rbrace$, is a pair  $(u,v)$ of biholomorphic maps
$$
u \colon C_1 \to C_2 \qquad \makebox{and} \qquad   v \colon \mathbb P^1 \to \mathbb P^1 
$$
such that $v(\mathcal B) = \mathcal B$ and $u \circ  f_1 = f_2 \circ v$. 
%If $v$ is the identity, then the covering isomorphism is called a \emph{strict covering isomorphism}. 
\end{definition}

\begin{rem}\label{twistII}
Let $\psi_i \colon G \to \Deck(f_i)$ be the corresponding $G$-actions, then the existence of a twisted covering isomorphism is equivalent to 
the existence of an automorphism $\alpha \in \Aut(G)$ and a biholomorphism $u\colon C_1 \to C_2$ such that 
\[
\psi_2(\alpha(g))\circ u = u \circ \psi_1(g) \quad \makebox{for all} \quad g \in G.
\]
As we shall see, this holds if and only if the corresponding generating triples belong to the same orbit of a certain group action on the set $\mathcal S(G)$ of all hyperbolic generating triples of $G$.
First of all, there  is a natural action of the  \emph{Artin-Braid} group  
$$
\mathcal B_3:=\langle \sigma_1,\sigma_2 ~ \vert ~ \sigma_1\sigma_2\sigma_1 =\sigma_2 \sigma_1 \sigma_2 \rangle 
$$
on $\mathcal S(G)$ defined by: 
\[
\sigma_1 \ast [a,b,c]:= [aba^{-1},a,c] \qquad \makebox{and} \qquad  \sigma_2 \ast [a,b,c]:= [a,bcb^{-1},b].
\]
This action commutes with the diagonal action of an automorphism  $\alpha \in \Aut(G)$ given by
$$\alpha \ast [a,b,c]:=[\alpha(a) , \alpha(b) ,\alpha(c) ].$$
Thus, we obtain a well-defined action of  $\Aut(G) \times \mathcal B_3$  on  $\mathcal S(G)$ by
\[
(\alpha,\delta)\ast S:=\alpha \ast (\delta \ast S).
\] 
\end{rem}

\begin{Proposition} \cite[Proposition 2.3]{IFG}\label{CoveringIsomorphisms}
Let $G$ be a finite group and $S,S' \in \mathcal S(G)$ be two generating triples of $G$. Then, the following are equivalent:
\begin{enumerate}
\item
There is a twisted covering isomorphism between $f_S\colon C_S \to \mathbb P^1$ and $f_{S'}\colon C_{S'} \to \mathbb P^1$.
\item
The generating triples $S$ and $S'$ are in the same $\Aut(G) \times \mathcal B_3$ orbit.
\end{enumerate} 
\end{Proposition}

\begin{rem}\label{ReCond2} 
    Theorem \ref{CoveringIsomorphisms} tells us that the collection of triangle $G$-covers modulo twisted covering isomorphisms is in  bijection with the quotient 
    \[
  \mathcal T(G):=   \mathcal{S}(G)/(\Aut(G)\times\mathcal{B}_3).
    \] 
    For this reason, several authors put effort into the development of an efficient algorithm to  compute these quotients, see \cite{CGP23} and \cite{Pa23}.
    In   \cite{CGP23} a database is set up, which contains a  representative for any orbit of hyperbolic generating triples for a fixed genus $g\leq 64$ and group order $d\leq 2000$. 
\end{rem}

\begin{Proposition}\label{biholliftkernel}
Let $X$ and $X'$ be Beauville $n$-folds given by the unmixed Beauville structures  $[S_1,\ldots,S_n]$ and $[S'_1,\ldots, S'_n]$, where $S_i \in S(G/K_i)$ and $S_i' \in S(G/K_i')$.  
Then $X$ and $X'$  
are biholomorphic if and only if
there exist an automorphism $\alpha\in\Aut(G)$, braids  $\delta_1,\ldots,\delta_n\in \mathcal B_3$ and a permutation $\tau\in \mathfrak S_n$, such that 
        \[
        K_i'=\alpha(K_{\tau(i)}) \qquad \makebox{and} \qquad 
            S'_i= (\overline{\alpha},\delta_i)\ast S_{\tau(i)}, \qquad \makebox{where} \qquad \overline{\alpha} \colon G/K_{\tau(i)} \to G/K_i'.
        \]
\end{Proposition}

\begin{proof}
 By uniqueness of the minimal realization (Theorem \ref{minimalreal}), every biholomorphism $f\colon X\to X'$ between the quotients lifts to a biholomorphism $$\hat{f}\colon C_{S_1}\times\ldots\times C_{S_n}\to C_{S'_1}\times\ldots\times C_{S'_n}.$$
As explained in Proposition \ref{rigidity-lemma}, this map must be of the form 
           \[
            \hat{f}(z_1,\ldots,z_n)=(u_1(z_{\tau(1)}),\ldots,u_n(z_{\tau(n)})),
        \]
for some permutation  $\tau\in\mathfrak S_n$. 
Such a map descends to the quotient level if and only if there exists an automorphism $\alpha\in\Aut(G)$ such that 
\begin{equation}\label{twistlift}
\psi_i'(\alpha(g))\circ u_i = u_i \circ \psi_{\tau(i)}(g) \qquad \makebox{for all} \qquad g \in G.
\end{equation}
 Now the action 
$
\psi_{\tau(i)} \colon G \to \Aut(C_{S_{\tau_i}})$
factors through the faithful action  $$\overline{\psi}_{\tau(i)} \colon G/K_{\tau(i)} \to \Aut(C_{S_{\tau_i}})$$ and similarly $\psi_i'$ descends to 
$\overline{\psi'}_{i} \colon  G/K_i' \to \Aut(C_{S_i'})$. Thus the Equation \ref{twistlift} from above is equivalent to 
\[
\overline{\psi_i'}(\overline{\alpha}(\overline{g}))\circ u_i = u_i \circ \overline{\psi}_{\tau(i)}(\overline{g}), \qquad \makebox{where} \qquad \overline{g} \in G/K_{\tau(i)}.  
\]
Here, $
\overline{\alpha} \colon G/K_{\tau(i)} \to G/K_{i}'$ denotes the induced isomorphism. 
In other words, we have twisted covering isomorphisms 
\[
    \begin{xy}
    \xymatrix{
    C_{S_{\tau_i}} \ar[r]^{u_i} \ar[d] & C_{S'_i} \ar[d] \\
   \displaystyle{ \frac{C_{S_{\tau_i}}}{G/K_{\tau(i)}} } \ar[r]_{v_i} & \displaystyle{ \frac{C_{S'_i}}{G/K_{i}'}}    }
    \end{xy}
\]
The claim follows now from  Proposition \ref{CoveringIsomorphisms}, that we use to translate into the language of spherical generating triples. 
The converse direction is clear. 
\end{proof}

\begin{remark}\label{GeneralQuotient}
The group $\Aut(G)\times (\mathcal B_3 \wr \mathfrak S_n)$ acts on $\mathcal{UB}_n(G)$ by 
\[
\big(\alpha, \delta_1, \ldots, \delta_n, \tau \big)\ast (S_1, \ldots,S_n):= \big( (\overline{\alpha},\delta_1)\ast S_{\tau^{-1}(1)}, \ldots, (\overline{\alpha},\delta_n)\ast S_{\tau^{-1}(n)} \big), 
\]
where $\overline{\alpha} \colon G/K_{i} \to G/\alpha(K_{i})$ are the induced isomorphisms. 
According to Proposition \ref{biholliftkernel}, the orbits of this action are in one-to-one correspondence with the biholomorphism classes of Beauville $n$-folds with group $G$. 
\end{remark}
To avoid a complicated notation, we will from now on 
denote $\overline{\alpha}$  by $\alpha$.

\section{The orbits of $\mathcal {UB}_n(G)$ modulo $\Aut(G)\times \mathcal (B_3 \wr \mathfrak S_n)$}\label{sec: computation_o_the_fibre}

\noindent 
 In this section, we will provide a detailed explanation on how to count the number of elements in the quotient
\[
Beau_n(G) := \frac{\mathcal {UB}_n(G)}{\Aut(G)\times (\mathcal B_3 \wr \mathfrak S_n)} \quad \makebox{by utilizing} \quad 
\mathcal T(G/K_i)=\frac{\mathcal S(G/K_i)}{\Aut(G/K_i) \times \mathcal B_3}.\]
As discussed in Remark \ref{ReCond2}, the latter may be determined by using the database from \cite{CGP23}. 
We break down the problem into several steps. First we exploit the natural action of $\Aut(G)$ on the set of potential kernels 
\[
\lbrace (K_1, \ldots, K_n) ~ \big\vert ~ K_1 \cap \ldots \cap \widehat{K}_i \cap \ldots \cap K_n=\lbrace 1_G\rbrace, ~~  K_i \trianglelefteq G  \rbrace,  
\]
which is defined by 
$$\alpha \ast (K_1, \ldots, K_n):= \left(\alpha(K_1), \ldots, \alpha(K_n)\right).$$ 
For each orbit, let  $
\mathcal K= (K_1, \ldots, K_n)$ be a representative. We use the permutations in $\mathfrak S_n$ to put equal kernels side by side. This allows us to assume and write
$\mathcal K= (K_1, \ldots, K_n)=(N_1^{n_1}, \ldots, N_k^{n_k})$
 with pairwise distinct $N_i$.
 Let $\Aut_{\mathcal K}(G)$ be the stabilizer of $\mathcal K$, then we can form the quotients 
\begin{equation}\label{quotientsK}
Beau_n(G, \mathcal K):= \dfrac{\prod_{i=1}^k\mathcal S(G/N_i)^{n_i}}{\Aut_{\mathcal K}(G)\times \prod_{i=1}^k (\mathcal B_3 \wr \mathfrak S_{n_i})} 
\end{equation}
 Running over all $\mathcal K$, 
 we obtain the elements of  $Beau_n(G)$,
 by selecting the classes in 
 $Beau_n(G, \mathcal K)$ such that the condition on the intersection of the stabilizer sets 
from Remark \ref{ReCond} (3) holds. 
 Thus it suffices to achieve a description of   $Beau_n(G, \mathcal K)$ in terms of $\mathcal T(G/N_i)$. 
 
 \begin{remark}\
 \begin{enumerate}
 \item 
  Since any automorphism  in $\Aut_{\mathcal K}(G)$ induces an automorphism in  
 $\Aut\left( G/N_i \right)$,
there is a natural surjective map 
 \[
 \eta \colon  Beau_n(G, \mathcal K) \to 
\frac{\prod_{i=1}^k \mathcal T(G/N_i)^{n_i}}{\prod_{i=1}^k \mathfrak S_{n_i}}
 \]
 Following \cite{fede24}, we solve the problem of counting \ref{quotientsK} by counting the elements in the 
fibres of $\eta$.
\item Dropping  the action of the symmetric group, we obtain another surjection 
\[
\pi \colon \frac{\prod_{i=1}^n S(G/K_i)}{\Aut_{\mathcal K}(G)\times \mathcal B_3^n} \to \prod_{i=1}^n  \mathcal T(G/K_i).
\]
\item 
For a point 
$x=\big([S_1], \ldots, [S_n]\big) \in \prod_{i=1}^n  \mathcal T(G/K_i)$
the fibers $\eta^{-1}\big([x]\big)$ and $\pi^{-1}(x)$ are related modulo  the action of the 
stabilizer of $x$. 
More precisely 
\[
\Stab(x) \leq \prod_{i=1}^k \mathfrak S_{n_i} \qquad \makebox{acts on} \qquad   \pi^{-1}(x)
\]
and we obtain a bijection $\eta^{-1}\big([x]\big) 
\simeq \pi^{-1}(x)/\Stab(x)$.
\end{enumerate}
\end{remark}
Thus we can break down the enumeration problem of the fibres of $\eta$ in   two steps:
\begin{enumerate}
    \item Describe the fibres $\pi^{-1}(x)$.
    \item Count the orbits of the $\Stab(x)$-action on $\pi^{-1}(x)$. 
\end{enumerate}
We start with the first step. By construction of $\pi$, the  fibre 
 is given by 
\[
\pi^{-1}(x)= \lbrace [\alpha_1 \ast  S_1, \ldots, \alpha_n \ast  S_n] ~ \big\vert ~ \alpha_i \in  \Aut(G/K_i) \rbrace. 
\]
The problem is that  different tuples of automorphisms 
\[
(\alpha_1, \ldots, \alpha_n), ~ (\beta_1, \ldots, \beta_n) \in \prod_{i=1}^n \Aut(G/K_i)
\]
may lead to same point in the fibre $\pi^{-1}(x)$. 
To deal with this ambiguity, we use the concept of 
automorphisms of \textit{braid type},  introduced in \cite{fede24}:  
\begin{definition}
	Let $S$ be a generating triple of  the finite group $G$. Then the group of automorphisms of \textit{braid type} on $S$  is defined as
	\[
	\mathcal B\Aut(G, S):=\lbrace \varphi \in \Aut(G)\colon \exists \  \sigma \in \mathcal{B}_3\ \textrm{such that }\varphi \ast S=\sigma\ast S  \rbrace.
	\]
\end{definition}
\begin{rem}
 Since the action of an automorphism of $G$ commutes with the action of a braid,  it follows that $\mathcal B\Aut(G, S)$ is a subgroup of $\Aut(G)$.
\end{rem}

\begin{Proposition}\label{BraidTypeIso}
Two  tuples of automorphisms 
\[
(\alpha_1, \ldots, \alpha_n), ~ (\beta_1, \ldots, \beta_n) \in \prod_{i=1}^n \Aut(G/K_i)
\]
yield the same point in the fibre $\pi^{-1}(x)$ if and only if there exists 
$\alpha \in \Aut_{\mathcal K}(G)$ and $\gamma_i \in \mathcal B\Aut(G/K_i, S_i)$ such that 
\[
 \beta_i=\alpha \circ \alpha_i \circ \gamma_i  \qquad \makebox{for all} \qquad i=1, \ldots, n.  
\]
\end{Proposition}

\begin{proof}
    Assume that 
    $$ [\alpha_1 \ast  S_1, \ldots, \alpha_n \ast  S_n]=
      [\beta_1 \ast  S_1, \ldots, \beta_n \ast  S_n] \in \pi^{-1}(x).$$
    Then there exists an automorphism $\alpha \in \Aut_{\mathcal K}(G)$ and braids $\delta_i \in \mathcal B_3$ such that 
    \[
    \beta_i \ast S_i= (\alpha \circ \alpha_i)  \ast (\delta_i \ast S_i), \qquad \makebox{so that} \qquad     
    (\alpha_i^{-1}\circ \alpha^{-1}\circ \beta_i) \ast S_i= \delta_i \ast S_i. 
    \]
    This shows that $\gamma_i :=\alpha_i^{-1}\circ \alpha^{-1}\circ \beta_i$ belongs to 
    $\mathcal B\Aut(G/K_i, S_i)$. Assume conversely that $\alpha \in \Aut_{\mathcal K}(G)$ and $\gamma_i \in \mathcal B\Aut(G/K_i, S_i)$ such that $
 \beta_i=\alpha \circ \alpha_i \circ \gamma_i$. Let $\delta_i \in \mathcal B_3$ be a braid fulfilling  $\gamma_i \ast S_i= \delta_i \ast S_i$, then we have 
\[
    \beta_i \ast S_i= (\alpha \circ \alpha_i)  \ast (\delta_i \ast S_i). 
   \]
   This shows that the corresponding points in the fibre agree. 
\end{proof}

\begin{rem}\label{remAction}
We have obtained a generalization of \cite[Thm. 2.18 and Cor. 2.20]{fede24} in higher dimension $n> 2$. More precisely, the following hold:
\begin{enumerate}
    \item 
    The group 
    \[
    \Aut_{\mathcal K}(G) \times  \prod_{i=1}^n\mathcal B\Aut(G/K_i, S_i) \qquad \makebox{acts on} \qquad \prod_{i=1}^n \Aut(G/K_i)
    \]
    via the rule 
    \[
    (\alpha,\gamma_1, \ldots, \gamma_n) \ast (\alpha_1, \ldots, \alpha_n) :=(\alpha \circ \alpha_1 \circ \gamma_1^{-1}, \ldots,\alpha \circ \alpha_n \circ \gamma_n^{-1})
    \]
By Proposition \ref{BraidTypeIso}, the quotient which is denoted by $
Q\left(\prod_{i=1}^n \Aut(G/K_i)\right)_{S_1,\dots, S_n}$,
is in bijection with $\pi^{-1}(x)$ via
\[
\psi \colon Q\left(\prod_{i=1}^n \Aut(G/K_i)\right)_{S_1,\ldots, S_n} \to \pi^{-1}(x), \quad [(\alpha_1, \ldots, \alpha_n)]   \mapsto   [\alpha_1 \cdot  S_1, \ldots, \alpha_n \cdot  S_n]
\]
By definition the bijection $\psi$ depends on the choices of representatives $S_i$ of the classes $[S_i]$.  
\item 
To  achieve a description of the fibers of $\eta$, we need to understand the induced action of $\Stab(x)$ on the quotient defined  in (1). 
Let 
\[
x=(x_1, \ldots, x_k) \in \prod_{i=1}^k  \mathcal T(G/N_i)^{n_i},  
\]
then up to exchanging the order of the factors within the product $T(G/N_i)^{n_i}$,  we may assume that 
\[
x_i=\big([S_{i,1}]^{m_{i,1}}, \ldots , [S_{i,l_i}]^{m_{i,l_i}}\big) \in \mathcal T(G/N_i)^{n_i}
\]

\medskip
\noindent 
with  pairwise distinct classes 
\[
[S_{i,1}], \ldots, [S_{i,l_i}] \qquad \makebox{and} \qquad m_{i,1} + \ldots + m_{i,l_i}=n_i.
\] 
In this notation, the  stabilizer of $x$ is given by 
\[
\Stab(x)= \prod_{i=1}^k\bigg(\prod_{j=1}^{l_i} \mathfrak S_{m_{i,j}}\bigg) <  \prod_{i=1}^k 
\mathfrak S_{n_i}.
\]
Clearly, the natural action of $\tau\in \Stab(x)$ on $\pi^{-1}(x)$ induces an action on the quotient 
\[%\label{eq: specific_Quotient}
    Q\left(\prod_{i=1}^k \Aut(G/N_i)^{n_i}\right)_{S_{1,1}^{m_{1,1}},\dots, S_{1,l_1}^{m_{1,l_1}},\dots, S_{k,1}^{m_{k,1}},\dots, S_{k,l_k}^{m_{k,l_k}}}
\] 
which  is compatible with $\psi$ and defined by the following rule:
\[%\label{eq: action_of_stab}
    \tau \ast [(\alpha_1, \ldots, \alpha_n)]:=[(\alpha_{\tau^{-1}(1)}, \ldots, \alpha_{\tau^{-1}(n)}]. 
\] 
This action is of course  only well defined if we choose the same representative for all of the  identical classes that occur in $x$. 
%More precisely, the elements in $\Stab(x)$ belonging to $\mathfrak S_{n_i}$ will only act on the classes of the factors $\Aut(G/N_i)^{n_i}$. 
\end{enumerate}
\end{rem}

\begin{rem}\label{rem: wrote_a MAGMA_script}
Following the approach presented in this section,  we wrote a MAGMA script that takes as input a group $G$, a sequence of kernels $K=(K_1, \dots, K_n)$ of $G$ and a sequence of classes  
\[
x=([S_1], \dots, [S_n]) \in \prod_{i=1}^n  \mathcal T(G/K_i).
\]
 It returns the fibre $\eta^{-1}([x])$, namely all unmixed Beauville $n$-folds with group $G$ and kernels $K$  defined by $[S_1], \dots, [S_n]$.
\end{rem}

\bigskip

\section{The Beauville dimension of a group}\label{section: BeauvilleDim}

In this section we discuss the Beauville dimension of a finite group $G$, that has been introduced by Carta and Fairbairn in \cite{CF22}. In their article Beauville manifolds of unmixed type and their groups  are studied  from a combinatorial and group theoretic perspective under the assumption that the $G$-action is absolutely faithful. 
They define the Beauville dimension of a finite group  $G$ as   the minimal length of a sequence $[S_1,\ldots,S_n]$ of generating triples for $G$, such that 
\[
 \Sigma_{S_1} \cap \ldots \cap  \Sigma_{S_n} =\lbrace 1_G \rbrace. 
  \]
If no such sequence exists then the Beauville dimension of $G$ is set to be $1$, see  \cite[Definition 1.3]{CF22}. 
According to their definition, the groups of Beauville dimension two are precisely the classical Beauville surface groups.
 Using computers  
they determine all finite groups $G$ of order less than or equal to $1023$ with Beauville dimension $2$, $3$ and $4$. 

\begin{Remark}
In their definition Carta and Fairbairn do not assume that the generating triples are hyperbolic i.e.  they also allow elliptic curves and even projective lines in the product. Hence the Beauville manifolds 
corresponding to the groups in their tables with Beauville  dimension 3 and 4 might not be of general type. 
In this cases the rigidity of the corresponding quotient manifolds does not follow 
from the condition that the curves are triangle curves, because there are further rigidity conditions involved, see \cite[Corollary 1.4]{BGK25}. 
Since they claim the rigidity in \cite[Definition 4]{CF22} but do not verify these extra conditions,  we decided to investigate their groups in 
 detail using the computer algebra system MAGMA: 
\begin{enumerate}
\item None of the $27$ groups of  Beauville dimension 3  admit a rigid action on an elliptic curve or on $\mathbb P^1$, i.e. all of them correspond to a Beauville threefold of general type.
\item There are $153$ groups of Beauville dimension $4$.  Non of them admit a rigid action on $\mathbb P^1$ and only $43$ of them a rigid action on an elliptic curve. The elliptic curve  is always  the Fermat cubic curve  and the branching signature is always equal to   
$[3,3,3]$. None of these 43 groups correspond to a Beauville 4-fold of general type, Kodaira dimension  $1$ or  $2$. 
\begin{enumerate}
\item Two out of the $43$ groups,  namely $\mathbb Z_3^2$ and $\rm{He}(3)$,  yield Beauville 4-folds of Kodaira dimension  $0$. They are quotients of a product 
$E_1\times E_2 \times E_3 \times E_4$ of four Fermat elliptic curves  $E_i=E$. 
In this case the extra rigidity condition 
\[
\big(H^1(\omega_{E_i}^{\otimes2})\otimes H^0(\omega_{E_j})\big)^G=0  \qquad \makebox{for all} \qquad i \neq j 
\]
can be fulfilled for both groups  $G=\mathbb Z_3^2$ and $G=\rm{He}(3)$ by suitable choices of generating triples. 
According to  \cite[Theorem 1.7]{BG21} there is a unique   Beauville 4-fold for each of these groups. The two 
 4-folds are  topologically distinct. 
\item The remaining $41$ groups  yield  Beauville $4$-folds 
\[
X=(E\times C_1\times C_2\times C_3)/G
\]
 of Kodaira dimension $3$, i.e. the curves $C_i$ have genus at least two and $E$ is the Fermat cubic curve as explained above. 
The extra rigidity condition  to be checked is 
\[
\big(H^1(\omega_{E}^{\otimes2})\otimes H^0(\omega_{C_i})\big)^G=0, \qquad \makebox{for all} \qquad 1\leq i \leq 3. 
\]
By \cite[Corollary 3.10]{BGK25} this condition holds true since 
$E\to E/G\simeq \mathbb P^1$ is branched  with signature  $[3,3,3]$. 
Below is a table with a structural description of these groups as subgroups of $\Aut(E)$, i.e. as 
semidirect products of the form $A \rtimes \mathbb Z_3$, where $A$ is an abelian group of translations and $\mathbb Z_3$ acts as a group of  rotations. 
The column $\rm{Id}$ contains the  MAGMA identifier: $\langle a, b \rangle$ denotes the b-th group of order $a$ in the MAGMA \emph{Database of Small Groups}. 

\begin{table}[!ht]
\begin{tabular}{cc}
 {\footnotesize
\begin{tabular}{| l | l | l |}
\hline
\rm{No.} & $G$ & $\rm{Id}$    \\
\hline
1 & $(\mathbb Z_2\times \mathbb Z_6)\rtimes \mathbb Z_3 $ & $ \langle 36, 11 \rangle$  \\
 \hline
2 & $ \mathbb Z_{21} \rtimes \mathbb Z_3$ & $\langle 63, 3 \rangle$   \\
 \hline
3 & $(\mathbb Z_3\times  \mathbb Z_9) \rtimes  \mathbb Z_3$ & $ \langle 81, 9 \rangle$   \\
\hline
4 & $\mathbb Z_6^2\rtimes  \mathbb Z_3$ & $ \langle 108, 22 \rangle $  \\
     \hline  
5 & $ \mathbb Z_{39} \rtimes  \mathbb Z_3$ & $\langle 117, 3 \rangle$  \\
 \hline
6 & $ (\mathbb Z_4 \times \mathbb Z_{12})\rtimes  \mathbb Z_3 $ & $  \langle 144, 68 \rangle$   \\
\hline
7 & $ \mathbb Z_{57} \rtimes  \mathbb Z_3 $ & $\langle 171, 4 \rangle $   \\
\hline
8 & $ ( \mathbb Z_{3}\times  \mathbb Z_{21})\rtimes  \mathbb Z_3$ & $\langle 189, 8 \rangle $   \\
 \hline
9 & $ (\mathbb Z_3 \times \mathbb Z_{15}) \rtimes \mathbb Z_3$ & $ \langle 225, 5 \rangle$   \\
\hline
10 & $\mathbb Z_9^2\rtimes  \mathbb Z_3$ & $ \langle 243, 26 \rangle$   \\
    \hline  
11 & $ (\mathbb Z_2 \times  \mathbb Z_{42}) \rtimes  \mathbb Z_3$ & $ \langle 252, 40 \rangle$    \\
 \hline
12 & $ \mathbb Z_{93} \rtimes  \mathbb Z_3$ & $\langle 279, 3 \rangle$   \\
\hline
13 & $( \mathbb Z_{6}\times  \mathbb Z_{18})\rtimes  \mathbb Z_3$ & $ \langle 324, 50 \rangle $   \\
\hline
14 & $ \mathbb Z_{111}\rtimes  \mathbb Z_3$ & $\langle 333, 4 \rangle$  \\
\hline
15 & $( \mathbb Z_{3}\times  \mathbb Z_{39})\rtimes  \mathbb Z_3$ & $ \langle 351, 8 \rangle $   \\
\hline
16 & $  \mathbb Z_{129} \rtimes  \mathbb Z_3$ & $\langle 387, 3 \rangle $ \\
\hline  
 17 & $\mathbb Z_{12}^2 \rtimes \mathbb Z_3 $ & $ \langle 432, 103 \rangle$    \\
 \hline
18 & $ \mathbb Z_{147} \rtimes  \mathbb Z_3$ & $\langle 441, 3 \rangle$   \\
\hline
19 & $ (\mathbb Z_7\times \mathbb Z_{21}) \rtimes  \mathbb Z_3$ & $ \langle 441, 12 \rangle $   \\
\hline
20 & $ (\mathbb Z_2 \times \mathbb Z_{78}) \rtimes  \mathbb Z_3$ & $\langle 468, 49 \rangle$  \\
   \hline   
21 & $( \mathbb Z_{3}\times  \mathbb Z_{57})\rtimes  \mathbb Z_3$ & $ \langle 513, 9 \rangle $   \\
\hline
 \end{tabular}}
&
 {\footnotesize
\begin{tabular}{| l | l | l |}
\hline
\rm{No.}& $G$ & $\rm{Id}$   \\
\hline
22 & $\mathbb Z_{183} \rtimes \mathbb Z_3$ & $\langle  549, 3 \rangle$   \\
     \hline  
 23   & $(\mathbb Z_{3}\times \mathbb Z_{63}) \rtimes \mathbb Z_3 $ & $ \langle 567, 13  \rangle$    \\
 \hline
24 & $(\mathbb Z_{8}\times \mathbb Z_{24}) \rtimes \mathbb Z_3 $ & $\langle 576, 1070 \rangle$   \\
\hline
25 & $\mathbb Z_{201} \rtimes \mathbb Z_3$ & $ \langle  603, 3 \rangle $  \\
\hline
26 & $\mathbb Z_{219} \rtimes \mathbb Z_3$ & $\langle 657, 4  \rangle$   \\
 \hline
27 & $\mathbb Z_{15}^2 \rtimes \mathbb Z_3$ & $ \langle  675, 12 \rangle $   \\
 \hline 
 28 & $(\mathbb Z_{2}\times \mathbb Z_{114}) \rtimes \mathbb Z_3$ & $\langle  684, 45  \rangle$   \\
\hline
29   & $\mathbb Z_{237} \rtimes \mathbb Z_3$ & $ \langle   711, 3 \rangle$   \\
 \hline
30 & $(\mathbb Z_{9}\times \mathbb Z_{27}) \rtimes \mathbb Z_3$ & $\langle  729, 95 \rangle$  \\
\hline
31 & $(\mathbb Z_{6}\times \mathbb Z_{42}) \rtimes \mathbb Z_3$ & $ \langle  756, 117  \rangle $  \\
\hline
32 & $\mathbb Z_{273} \rtimes \mathbb Z_3$ & $\langle  819, 9  \rangle$   \\
\hline
33 & $\mathbb Z_{273} \rtimes \mathbb Z_3$ & $ \langle  819, 10  \rangle $   \\
\hline
34 & $(\mathbb Z_{3}\times \mathbb Z_{93}) \rtimes \mathbb Z_3$ & $\langle  837, 8 \rangle$   \\
 \hline  
35  & $\mathbb Z_{291} \rtimes \mathbb Z_3 $ & $ \langle   873, 3  \rangle$    \\
 \hline
36 & $(\mathbb Z_{10}\times \mathbb Z_{30}) \rtimes \mathbb Z_3$ & $\langle  900, 141  \rangle$   \\
\hline
37 & $\mathbb Z_{309} \rtimes \mathbb Z_3 $ & $ \langle  927, 3  \rangle $   \\
\hline
38 & $\mathbb Z_{18}^2 \rtimes \mathbb Z_3$ & $\langle   972, 122  \rangle$   \\
\hline
39 & $\mathbb Z_{327} \rtimes \mathbb Z_3 $ & $ \langle  981, 4  \rangle $   \\
\hline
40 & $(\mathbb Z_{3}\times \mathbb Z_{111}) \rtimes \mathbb Z_3$ & $\langle  999, 9  \rangle$  \\
    \hline  
41 & $(\mathbb Z_{4}\times \mathbb Z_{84}) \rtimes \mathbb Z_3$ & $\langle  1008, 409  \rangle$   \\
     \hline  
     && \\
     \hline
 \end{tabular}}
\end{tabular}
\end{table}

\end{enumerate}
In summary our verifications show that all of the groups in the tables of Carta and Fairbairn yield Beauville manifolds. 
We suggest to  incorporate the additional rigidity conditions into the definition of the Beauville dimension, or alternatively, restrict to hyperbolic generating triples. Moreover, we want to point out that the uniqueness of the minimal realization, which in particular allows us to attach a unique group to a given Beauville manifold $X$ is only established if the Kodaira dimension is maximal. In case where $X$ has Kodaira dimension zero, we can uniquely attach the holonomy group of the underlying flat Riemannian manifold, which is a quotient of $G$.  The two $4$-folds in $(a)$ both have holonomy $\mathbb Z_3^2$, 
cf. \cite[Remark 5.10]{BG21}.   
\end{enumerate}
\end{Remark}

As we have seen,  there are Beauville manifolds of dimension  $\dim(X) \geq 3$ obtained by  actions which are not absolutely faithful. 
Thus it makes sense to generalize the definition of the Beauville dimension by allowing non-trivial kernels. 
In contrast to \cite{CF22} we decided to restrict to hyperbolic generating triples.

\begin{definition}
The \emph{Beauville dimension} $d(G)$ of a finite group  $G$ is the minimal positive integer $n \geq 2$, such that $\mathcal{UB}_n(G) \neq \emptyset$. 
If no such integer exists then the Beauville dimension of $G$ is set to be equal to $1$. 
\end{definition}
It is natural to  ask the question  if there are further groups of Beauville dimension $d(G) \geq 3$, using our definition. 
For this purpose we wrote a MAGMA algorithm to check if a given finite group $G$ admits an unmixed $n$-fold  Beauville structure. Since the presence of non-trivial kernels increases  the computational difficulty drastically, we restrict to groups of order less than or equal to $255$ and $n=3$. The algorithm also determines the Beauville dimension of the respective  groups according to our definition. We find:

\begin{Proposition}\label{Beauville255}
The groups $G$ of order less or equal to $255$, which admit an unmixed 3-fold Beauville structure are the following:

  \begin{table}[!ht]
\centering
 {\footnotesize
\begin{tabular}{| l | l | l | l |}
\hline
\rm{No.} & $G$ & $\rm{Id}$ & $d(G)$  \\
\hline
1 & $\mathbb Z_5^2$ & $\langle 25, 2 \rangle$& $2$  \\
 \hline
2 & $\mathbb Z_7^2$ & $\langle 49, 2 \rangle$& $2$  \\
 \hline
3 & $\mathfrak S_5$ & $ \langle 120, 34 \rangle$& $2$  \\
\hline
4 & $\mathbb Z_{11}^2$ & $\langle 121, 2 \rangle$& $2$  \\
\hline
5 & ${\rm He}(5)$ & $ \langle 125, 3 \rangle$& $2$  \\
\hline 
6 & $\mathbb Z_5^3$ & $  \langle 125, 5 \rangle$& $3$  \\
 \hline
7 & $\mathbb Z_2^4 . Q_8 $ & $\langle 128, 36 \rangle$& $2$  \\
\hline
8 & $\PSL(2,7) $ & $\langle 168, 42 \rangle$& $2$  \\
\hline
9 & $\mathbb Z_{13}^2$ & $\langle 169, 2 \rangle$& $2$  \\
\hline
10 & ${\rm SL}(2,5) \rtimes \mathbb Z_2 $ & $\langle 240, 90 \rangle$& $2$  \\
 \hline
11 & $ \mathcal A_5 \rtimes \mathbb Z_4  $ & $\langle 240, 91 \rangle$& $2$  \\
  \hline
12 & $  \mathbb Z_2 \times \mathfrak S_5  $ & $\langle 240, 189 \rangle$& $2$  \\
\hline
 13 & $\big(\mathbb Z_3 \times  {\rm He}(3)\big) \rtimes \mathbb Z_3  $ & $\langle  243, 3 \rangle$& $2$  \\
 \hline 
  14 & $\big(\mathbb Z_3 \times  M_{27}\big) \rtimes \mathbb Z_3  $ & $\langle  243, 4 \rangle$& $3$  \\
   \hline 
  15 & $\mathbb Z_3^3  \rtimes  \mathbb Z_9 $ & $\langle  243, 13 \rangle$& $3$  \\
     \hline  
 \end{tabular}}
\end{table} 
Here ${\rm He}(p)$ is the Heisenberg groups of order $p^3$ and $M_{27}$ is the unique non-abelian group of order $27$ which has an element of order $9$. 
\end{Proposition}

There are three groups of Beauville dimension $d(G) =3$ in our table. 
The smallest of them $\mathbb Z_5^3$ is the only abelian and moreover the only group  that does not appear in the table of Carta and Fairbairn. 
 This means that a Beauville structure of $\mathbb Z_5^3$ can only exist  with  non-trivial kernels. This   group is interesting for the following reasons:
\begin{remark} \
\begin{enumerate}
\item 
According to \cite{IFG} the  abelian  groups of Beauville dimension two are $\mathbb{Z}_n^2$, where  $\gcd(n, 6) = 1$.
 Carta and Fairbairn extended this result to higher dimensions. Using their definition,  they show that an abelian group of Beauville dimension greater than two has Beauville dimension 
 four and is isomorphic to $\mathbb Z_n^2$,  where  $\gcd(n, 2) = 1$, see \cite[Theorem 2.9]{CF22}. In contrast the group $\mathbb Z_5^3$ is not $2$-generated, which raises the 
 question of a general structure theorem of abelian groups $G$ of Beauville dimension $d(G) \geq 3$ in view of  our definition.
 \item 
Carta and Fairbairn point out that the order of the  groups of Beauville dimension greater than two   which are contained in their tables is always divisible by $3$. They ask if this is  a general fact \cite[Problem 4.3]{CF22}. Using our definition of Beauville dimension the group $\mathbb Z_5^3$ of Beauville dimension three gives a negative answer to this question. 
\end{enumerate}
 \end{remark}

\begin{Proposition}
The group $\mathbb Z_n^3$ has a
$3$-fold Beauville structure if and only if $\gcd(n, 6) = 1$.
\end{Proposition}

\newcommand*{\thead}[1]{\multicolumn{1}{|c|}{#1}}

\begin{proof}
Assume that $\gcd(n, 6) = 1$.
We choose the kernels $K_i=\langle e_i \rangle$ for $G = \mathbb Z_n^3 = \langle e_1,e_2,e_3 \rangle$ and the following generating triples of the quotient groups $G/K_i$ of type $[n,n,n]$: 

\begin{table}[!ht] \centering
 {\footnotesize
\begin{tabular}{| l | l | l |}
\hline
\thead{i} & $S_i$ & $\Sigma_{S_i} +  K_i$  \\
\hline
1 & $[e_2 - e_3,e_2 + e_3,-2 e_2]$ & $\lbrace (l,j,k) \in G ~ \vert ~ j=-k ~ \makebox{or} ~  j=k ~ \makebox{or} ~ j=0 \rbrace$ \\
\hline
2 & $[e_1 + e_3,e_3 - e_1,-2 e_3]$ & $\lbrace (l,j,k) \in G ~ \vert ~ l=k ~ \makebox{or} ~  l=-k ~ \makebox{or} ~ k=0 \rbrace $ \\
\hline
3 & $[2e_1+e_2,e_1,-3 e_1 - e_2]$ & $\lbrace (l,j,k) \in G ~ \vert ~ l=2j ~ \makebox{or} ~  l=0 ~ \makebox{or} ~ l=3j \rbrace $ \\
\hline
\end{tabular}}
\end{table} 
To verify the freeness condition, we consider an element $(l,j,k)\in \bigcap_{i=1}^3 (\Sigma_{S_i} + K_i) $ and show that it is trivial. If $j=0$, we conclude directly from the third row that $l=0$ and consequently from the second row $k=0$, since $\gcd(n,6)=1$. 
On the other hand, if $j \neq 0$, the three conditions yield $0 \neq j= \pm k = \pm l = \pm m \cdot j$ with $m \in \{ 2,3 \}$, leading to a contradiction. \\
Suppose that $G=\mathbb Z_n^3$ has a
$3$-fold Beauville structure. We show that $\gcd(n, 6) = 1$. \\
{\bf Step 1.} Reduction to the case  $\mathbb Z_{p^k}^3$.
The group $G$ is the product of its Sylow subgroups 
\[
G=\bigoplus_{p ~ prime } G_p, \qquad \makebox{with } \qquad G_p = \mathbb Z_{p^{k_p}}^3.
\]
This implies 
\[
G/K_i=\bigoplus_{p ~ prime } G_p/(K_i\cap G_p). 
\]

If the group $G/K_i$ has the generating triple $S_i:=[x_i,y_i,z_i]$ (from a Beauville structure of $G$), then  
$G_p/(K_i\cap G_p)$  has the generating triple $\pi_p(S_i)$ %$[(x_i)_p,(y_i)_p,(z_i)_p]$.  Here $x_p$, $y_p$ and $z_p$ denote 
consiting of the projections  $\pi_p(x_i)$, $\pi_p(y_i)$ and $\pi_p(z_i)$. Let $\Sigma_{S_i}$  and  $\Sigma_{\pi_p(S_i)}$ be the corresponding  stabilizer sets. We have  %\footnote{The elements $x_p$, $y_p$ and $z_p$ are multiples of $x$, $y$ and $z$.} 
$\Sigma_{\pi_p(S_i)} \subset \Sigma_{S_i}$ because  the projections $\pi_p(x_i)$, $\pi_p(y_i)$ and $\pi_p(z_i)$ are multiples of $x_i$, $y_i$ and $z_i$. Therefore, we have   
\[
\Sigma_{\pi_p(S_i)}+(K_i \cap  G_p)  \subset \Sigma_{S_i} + K_i. 
\]
Note that  $K_i \cap  G_p$ is a proper subgroup of  $G_p$. Otherwise, if  $K_1 \cap  G_p =G_p$, then  $K_2 \cap  G_p=0$ by minimality of the action and therefore  $$G_p/(K_2\cap G_p)=G_p= \mathbb Z_{p^{k_p}}^3.$$ A contradiction, since the group $\mathbb Z_{p^{k_p}}^3$ is not 2-generated. Thus we have shown that $\mathbb Z_{p^{k_p}}^3$ admits a 3-fold Beauville structure for all primes $p$ dividing $n$.  \\
{\bf Step 2.} Reduction to $\mathbb Z_{p}^3$. By the first step we may assume that $G=\mathbb Z_{p^k}^3$.
Since $G/K_i$ is 2-generated, it holds   
\[
G/K_i= \mathbb Z_{p^{a_i}} \times \mathbb Z_{p^{b_i}}  , \quad \makebox{where} \quad k \geq a_i \geq b_i \geq 0. 
\]
Here $b_i=0$ is possible, but $a_i \neq 0$ according to the remark at the end of the first step. 
Since 
$G/K_i$ is $2$-generated the kernel $K_i$ must have an element of order  $p^k$. We use this fact to show  $a_i=k$. Assume $a_1 <k$, then $p^{k-1}G \subset K_1$. Let  $g \in K_2$ be an element of order  $ord(g)=p^k$, then  
$$0\neq p^{k-1}g\in (p^{k-1}G) \cap K_2  \subset K_1 \cap K_2 =0.$$ A contradiction.  Thanks to this argument we have 
\[
G/K_i= \mathbb Z_{p^{k}} \times \mathbb Z_{p^{b_i}}  , \quad \makebox{with} \quad k \geq b_i \geq 0. 
\]
This shows that 
\[
\frac{p^{k-1}G}{ (p^{k-1}G) \cap K_i} \simeq p^{k-1} \cdot (G/K_i) \quad \simeq \quad  \mathbb Z_p \quad \makebox{or} \quad  \mathbb Z_p^2.
\]
Now let  $S_i=[x_i,y_i,z_i]$ be a generating triple of $G/K_i$, then  the elements of the triple 
\[
[p^{k-1}x_i,p^{k-1}y_i,p^{k-1}z_i]
\]
 generate the group  
$p^{k-1} \cdot (G/K_i)$. This generating triple  has at least 2 elements different from zero.  %However one element in the triple  might  be zero. %In this case $p^{k-1} \cdot (G/K_i)$ is cyclic and therefore isomorphic to $\mathbb Z_p$. 
It holds 
\[
(p^{k-1} \cdot \Sigma_{S_i})+ (K_i \cap p^{k-1}G) \subset \Sigma_{S_i} + K_i.
\]
Here the set  $(p^{k-1} \cdot \Sigma_{S_i})$ coincides with the stabilizer set of  
\[
[p^{k-1}x_i,p^{k-1}y_i,p^{k-1}z_i].
\]
A priori one of the elements of this triple might be zero, but it never happens. Assume otherwise and  $i=1$, then 
 $$(p^{k-1}\cdot \Sigma_{S_1})+(K_1\cap p^{k-1}G)=p^{k-1}G\cong \mathbb Z_p^3.$$
We observe that the intersection of the sets $(p^{k-1}\cdot \Sigma_{S_j})+(K_j\cap p^{k-1}G)$ for $j=2$ and $3$ must be trivial. This implies that \[\frac{p^{k-1}\cdot G}{(K_2\cap p^{k-1}G)+(K_3\cap p^{k-1}G)}\cong \mathbb Z_p
\]
is a Beauville surface group, which is a contradiction. 
In summary this shows that  $p^{k-1}\cdot G\cong \mathbb Z_p^3$ admits a 3-fold Beauville structure. 

{\bf Step 3.} It remains to exclude the groups  $G=\mathbb Z_p^3$ for the primes  $p=2$ and $p=3$.   
This is a straightforward computation  by hand  or by MAGMA.
\end{proof}

\bigskip

\section{Proof of the main theorems}\label{section: Examples}

 We will now apply our implementation of the method outlined in Section \ref{sec: computation_o_the_fibre} to proof the main theorems from the introduction.

\begin{proof}[Proof of Theorem \ref{THMI} and \ref{THMII}] 
According to Proposition \ref{Beauville255} we know that  $\mathcal U B_3(G)=\emptyset$ for all groups $G$ of order $\vert G \vert \leq 25$ except for $G=\mathbb Z_5^2=\langle e_1, e_2 \rangle$.  
Here the types of the generating triples $S_i$ defining $X$ are $[5,5,5]$.  By the Hurwitz formula, the possible genera are   either $6$ or $2$, depending on the kernel $K_i$ being  trivial or isomorphic to $\mathbb Z_5$. 
Using Proposition \ref{Global_Invar} we see that the only possible values of the holomorphic Euler number are $\chi(\mathcal{O}_X)=-1$ or $-5$. More precisely,  $\chi(\mathcal{O}_X) = -1$ occurs when two of the three kernels are trivial and   $\chi(\mathcal{O}_X) = -5$   when all kernels are trivial.
In the first case,   we can reorder $K_i$ in such a way that $K_1=K_2=\lbrace 0\rbrace$ and only the third one is not trivial. Furthermore, since the automorphism group $\GL(2,5)$ of $\mathbb Z_5^2$ is acting transitively on $\mathbb Z_5^2$, we can assume $K_3=\langle e_2\rangle$.
There is only one  $\GL(2,5)\times \mathcal B_3$ orbit  on the set of  spherical system of generators of $\mathbb Z_5^2$ with signature $[5,5,5]$. It is represented by $S:=[e_1,e_2,4e_1+4e_2]$. Similarly for  $\overline{G}_3\cong \mathbb Z_5$ the set 
$\mathcal T(\overline{G}_3)$ consists of only one equivalence class  represented by $S'=[e_1,e_1,3e_1]$.
Running  the script presented in Section \ref{sec: computation_o_the_fibre}  for $K=(\langle 0\rangle,\langle 0\rangle, \langle e_2 \rangle)$ and $S=[S,S, S']$ we obtain 8 biholomorphism classes of unmixed Beauville 3-folds with $\chi(\mathcal O_X)=-1$ represented by the  Beauville structures in Table \ref{tab: Beauv_chi=-1}.
%  \begin{table}[!ht]
%\centering
% {\footnotesize
%\begin{tabular}{|c|c|c|c|c|c|c|c|c|c|}
%\hline
 %& $S_1$ & $S_2$ & $S_3$ & $h^{3,0}$ & $h^{2,0}$ & $h^{1,0}$ & $h^{1,1}$ & $h^{2,1}$ \\
%\hline
%1 & $[e_1+e_2,e_2, 4e_1+3e_2]$ & $[3e_1,4e_1+e_2, 3e_1+4e_2]$ & $[4e_2,4e_2,2e_2]$ & 3 & 1 & 0 & 5 & 9 \\
%\hline
%2 & $2e_1+e_2,3e_1+e_2, 3e_2$ & $2e_1+2e_2,2e_1+3e_2, e_1$  & $2e_2,2e_2,e_2$ & 4 & 2 & 0 & 3 & 8  \\
%\hline
%3 & $3e_1+4e_2, 4e_1+3e_2, 3e_1+3e_2$ & $3e_2, 2e_1, 3e_1+2e_2$ & $3e_2,3e_2, 4e_2$ & 2 & 0 &  0& 7 & 10 \\
%\hline
%4 & $3e_1+4e_2, 4e_2,2e_1+2e_2$ & $3e_1+e_2, e_1+4e_2,e_1$ & $4e_2,4e_2,2e_2$ & 3 & 1 & 0 & 5 & 9 
%\\
%\hline
%5 & $4e_1+e_2, 2e_2, e_1+2e_2$ & $2e_1,e_1+3e_2, 2e_1+2e_2$ & $2e_2,2e_2,e_2$ & 2 & 0 & 0 & 7 & 10 \\
%\hline
%6 & $e_1+3e_2, e_1+e_2,3e_1+e_2$ & $2e_1+3e_2,2e_2,3e_1$ & $4e_2,4e_2,2e_2$ & 3 & 1 & 0 & 5 & 9  \\
%\hline
%7 & $4e_1+e_2, e_1,4e_2$ & $3e_1+3e_2,3e_1+4e_2,4e_1+3e_2$  & $2e_2,2e_2,e_2$ & 3 & 1 & 0 & 5 & 9 \\
%\hline
%8 & $4e_1+2e_2, e_1+2e_2, e_2$ & $3e_1+3e_2,4e_1,3e_1+2e_2$ & $4e_2,4e_2,2e_2$ & 4 & 2 & 0 & 3 & 8 \\
%\hline
%\end{tabular}}
%\caption{The $8$ biholomorphism classes of unmixed Beauville 3-folds with $G=\mathbb Z_5^2$ and $\chi(\mathcal %O_X)=-1$ and their Hodge numbers.}
%\label{tab: Beauv_chi=-1}
%\end{table}

  \begin{table}[!ht]
\centering
 {\footnotesize
\begin{tabular}{|l|l|l|l|l|l|l|l|l|l|}
\hline
 & $S_1$ & $S_2$ & $S_3$ & $h^{3,0}$ & $h^{2,0}$ & $h^{1,0}$ & $h^{1,1}$ & $h^{2,1}$ \\
\hline
1 & $[ e_2, 3e_1 + 3e_2, 2e_1 + e_2 ]$ & $[4e_1 + 3e_2 ,3e_1, 3e_1+2e_2]$ & $[3e_1,3e_1,4e_1]$ & 2 & 0 & 0 & 7 & 10 \\
\hline
2 & $[e_2,3e_1 + 3e_2,2e_1+e_2]$ & $[e_1+4e_2,3e_1+e_2,e_1]$  & $[e_1,e_1,3e_1]$ & 2 & 0 & 0 & 7 & 10  \\
\hline
3 & $[4e_1+3e_2, 4e_2,e_1+3e_2]$ & $[4e_1+e_2,4e_1+4e_2,2e_1]$ & $[2e_1,2e_1,e_1]$ & 3 & 1 & 0 & 5 & 9 \\
\hline
4 & $[3e_1+e_2,3e_1+4e_2,4e_1]$ & $[3e_2,4e_1+e_2,e_1+e_2]$ & $[4e_1,4e_1,2e_1]$ & 3 & 1 & 0 & 5 & 9 \\
\hline
5 & $[2e_1+2e_2,2e_1+4e_2,e_1+4e_2]$ & $[2e_2,4e_1,e_1+3e_2]$ & $[4e_1,4e_1,2e_1]$ & 3 & 1 & 0 & 5 & 9 \\
\hline
6 & $[e_1+e_2,e_1+2e_2,3e_1+2e_2]$ & $[2e_2, 4e_1,e_1+3e_2]$ & $[2e_1,2e_1,e_1]$ & 3 & 1 & 0 & 5 & 9  \\
\hline
7 & $[4e_1+4e_2,4e_1+3e_2,2e_1+3e_2]$  & $[e_2,2e_1,3e_1+4e_2]$ & $[e_1,e_1,3e_1]$ & 4 & 2 & 0 & 3 & 8 \\
\hline
8 & $[e_2,2e_1+2e_2,3e_1+2e_2]$ & $[e_1+4e_2,3e_1+e_2,e_1]$ & $[3e_1,3e_1,4e_1]$ & 4 & 2 & 0 & 3 & 8 \\ 
\hline
\end{tabular}}
\caption{The $8$ biholomorphism classes of unmixed Beauville 3-folds with $G=\mathbb Z_5^2$ and $\chi(\mathcal O_X)=-1$ and their Hodge numbers.}
\label{tab: Beauv_chi=-1}
\end{table}

The  Hodge numbers  are computed using the generating triples as explained in \cite[Theorem 3.7]{FG16}. 
It remains to classify the unmixed Beauville manifolds $X$  with 
$\chi(\mathcal O_X) \in \lbrace -5,-4,-3,-2 \rbrace$ 
under the assumption that the $G$-action is absolutely faithful. 
By Hurwitz's bound we have $\vert G \vert \leq 84(g_i-1)$, where $g_i=g(C_i)$. In combination with the formula for 
$\chi(\mathcal O_X)$ from Proposition \ref{Global_Invar} we obtain a bound for the group order in terms of the holomorphic Euler number: 
\[
N:=\vert G \vert \leq \lfloor 168 \sqrt{-21\chi(\mathcal O_X)} \rfloor. 
\]
This already shows the finiteness of the classification, which is performed with  MAGMA. 
To make the algorithm  more efficient, we invoke some additional combinatorics, as explained in \cite{FG16}. 
Since  $(g_i-1)$ is a divisor of  $N\cdot \chi(\mathcal O_X)$ we can create  for fixed value of $\chi=\chi(\mathcal O_X)$ a list of $4$-tuples for the possibilities of the group order and the genera 
\[
[N,g_1,g_2,g_3]. 
\]
 For each $4$-tuple we determine the possible types $T_i=[m_{i,1},m_{i,2},m_{i,3}]$ of the generating triples in the Beauville structures. 
The entries $m_{i,j}\geq 2$ are divisors of $N$, fulfill the Hurwitz formula (cf. Remark \ref{triangle}) and the following additional combinatorial constraints:
\[
m_{i,j} \vert (g_{[i+1]}-1)(g_{[i+2]}-1) \qquad \makebox{and} \qquad  m_{i,j} \leq 4g_i+2,  
\]
see \cite[Prop. 4.8]{FG16}. 
This allows us to determine a list of $4$-tuples  $$[N,T_1,T_2,T_3]$$ of possible group orders and types. Not all of them will occur. However, any 
Beauville structure 
$[S_1,S_2,S_3]$ attached 
to an unmixed Beauville threefold $X$ with  $\chi=\chi(\mathcal O_X)$ obtained from an absolutely faithful  $G$-action yields a tuple 
 $$[\vert G \vert, T(S_1),T(S_2),T(S_3)]$$ in this  list. For each $[N,T_1,T_2,T_3]$ we run through the groups $G$ of order $N$ and check if there is an 
 unmixed Beauville structure $[S_1,S_2,S_3]$, with $S_i\in \mathcal S(G)$ and $T_i=T(S_i)$. 
In case, we determine for each  $T_i$  a representative of each orbit of the $\Aut(G)\times \mathcal B_3$-action  on the set of generating triples  $S_i$ of  $G$ with type $T_i$. 
The method from  Section \ref{sec: computation_o_the_fibre} is then used to 
determine  the $\Aut(G)\times (\mathcal B_3 \wr \mathfrak S_3)$ orbits, i.e. the number $\mathcal N$ of biholomorphism classes. The output is sumarized in Table \ref{tab: classification}.
   
\begin{table}[!ht]
\centering
 {\footnotesize
\begin{tabular}{|l|l|l|l|l|l|l|l|l|l|l|l|}
\hline
 & $G$ & $T_1$ & $T_3$ & $T_2$ & $h^{3,0}$ & $h^{2,0}$ & $h^{1,0}$ & $h^{1,1}$ & $h^{1,2}$ & $\chi$ & $\mathcal N$ \\
\hline
1 & $\mathfrak S_5$ & $[2,5,4]$ & $[2,6,5]$ & $[3,4,4]$ &$4$  &  $1$& $0$ &  $5$& $12$ & $-2$& $1$\\
\hline
2 & $\PSL(2,7)$ & $[2,3,7]$ & $[3,3,4]$ & $[7,7,7]$ &$6$  & $1$ & $0$ & $11$ &$24$  & $-4$& $2$ \\
\hline
3 & $\PSL(2,7)$ & $[2,3,7]$ & $[3,3,4]$ & $[7,7,7]$ &$9$  & $4$ & $0$ & $5$ &$21$  & $-4$& $2$ \\
\hline
4 & $\PSL(2,7)$ & $[2,3,7]$ & $[4,4,4]$ & $[3,3,7]$ & $6$ &  $1$& $0$ & $7$ & $20$ & $-4$& $2$ \\
\hline
5 & $\PSL(2,7)$ & $[2,3,7]$ & $[4,4,4]$ & $[3,3,7]$ & $7$ &  $2$& $0$ & $5$ & $19$ & $-4$& $2$ \\
\hline
6 & $\mathfrak S_5$ & $[2,5,4]$ & $[3,4,4]$ & $[3,6,6]$ & $8$ &  $2$&  $0$&  $7$&  $24$& $-5$& $1$\\
\hline
7 & $\mathbb Z_5^2$ & $[5,5,5]$ & $[5,5,5]$ & $[5,5,5]$ & $6$ & $0$ & $0$ & $15$ & $30$& $-5$ & $2$\\
\hline
8 & $\mathbb Z_5^2$ & $[5,5,5]$ & $[5,5,5]$ & $[5,5,5]$ & $7$ & $1$ & $0$ & $13$ & $29$& $-5$ & $3$\\
\hline
9 & $\mathbb Z_5^2$ & $[5,5,5]$ & $[5,5,5]$ & $[5,5,5]$ & $7$ & $1$ & $0$ & $17$ & $33$& $-5$ & $1$\\
\hline
10 & $\mathbb Z_5^2$ & $[5,5,5]$ & $[5,5,5]$ & $[5,5,5]$ & $8$ & $2$ & $0$ & $11$ & $28$& $-5$ & $13$\\
\hline
11 & $\mathbb Z_5^2$ & $[5,5,5]$ & $[5,5,5]$ & $[5,5,5]$ & $8$ & $2$ & $0$ & $15$ & $32$& $-5$ & $3$\\
\hline
12 & $\mathbb Z_5^2$ & $[5,5,5]$ & $[5,5,5]$ & $[5,5,5]$ & $9$ & $3$ & $0$ & $9$ & $27$& $-5$ & $14$\\
\hline
13 & $\mathbb Z_5^2$ & $[5,5,5]$ & $[5,5,5]$ & $[5,5,5]$ & $9$ & $3$ & $0$ & $13$ & $31$& $-5$ & $4$\\
\hline
14 & $\mathbb Z_5^2$ & $[5,5,5]$ & $[5,5,5]$ & $[5,5,5]$ & $10$ & $4$ & $0$ & $7$ & $26$& $-5$ & $12$\\
\hline
15 & $\mathbb Z_5^2$ & $[5,5,5]$ & $[5,5,5]$ & $[5,5,5]$ & $10$ & $4$ & $0$ & $11$ & $30$& $-5$ & $8$\\
\hline
16 & $\mathbb Z_5^2$ & $[5,5,5]$ & $[5,5,5]$ & $[5,5,5]$ & $11$ & $5$ & $0$ & $5$ & $25$& $-5$ & $3$\\
\hline
17 & $\mathbb Z_5^2$ & $[5,5,5]$ & $[5,5,5]$ & $[5,5,5]$ & $11$ & $5$ & $0$ & $9$ & $29$& $-5$ & $7$\\
\hline
18 & $\mathbb Z_5^2$ & $[5,5,5]$ & $[5,5,5]$ & $[5,5,5]$ & $12$ & $6$ & $0$ & $3$ & $24$& $-5$ & $4$\\
\hline
19 & $\mathbb Z_5^2$ & $[5,5,5]$ & $[5,5,5]$ & $[5,5,5]$ & $12$ & $6$ & $0$ & $7$ & $28$& $-5$ & $3$\\
\hline
\end{tabular}}
\caption{Beauville 3-folds $X$ with $\chi(\mathcal O_X)\in \lbrace -5,-4,-3,-2 \rbrace $ obtained by an absolutely faithful action.}
\label{tab: classification}
\end{table}
\end{proof}


\begin{thebibliography}{xxxxxxxx}
\bib{BCG08}{article}{
    AUTHOR = {Bauer, I.},
  AUTHOR = {Catanese, F.},
 AUTHOR = {Grunewald, F.},
     TITLE = {The classification of surfaces with {$p_g=q=0$} isogenous to a
              product of curves},
   JOURNAL = {Pure Appl. Math. Q.},
  FJOURNAL = {Pure and Applied Mathematics Quarterly},
    VOLUME = {4},
      YEAR = {2008},
    NUMBER = {2},
     PAGES = {547--586},
      ISSN = {1558-8599,1558-8602},
   MRCLASS = {14J15 (14J29)},
  MRNUMBER = {2400886},
}
%
\bib{IFG}{incollection}{
    AUTHOR = {Bauer, I.},
     AUTHOR = {Catanese, F.},
     AUTHOR = {Grunewald, F.},
     TITLE = {Beauville surfaces without real structures},
 BOOKTITLE = {Geometric methods in algebra and number theory},
    SERIES = {Progr. Math.},
    VOLUME = {235},
     PAGES = {1--42},
 PUBLISHER = {Birkh\"auser Boston, Boston, MA},
      YEAR = {2005},
      ISBN = {0-8176-4349-4},
   MRCLASS = {14J29 (14J10 14J25 14L30)},
  MRNUMBER = {2159375},
}
%
\bib{BCP97}{incollection}{
    AUTHOR = {Bosma, W.},
    AUTHOR = {Cannon, J.},
    AUTHOR = {Playoust, C.},
     TITLE = {The {M}agma algebra system. {I}. {T}he user language},
      NOTE = {Computational algebra and number theory (London, 1993)},
   JOURNAL = {J. Symbolic Comput.},
  FJOURNAL = {Journal of Symbolic Computation},
    VOLUME = {24},
      YEAR = {1997},
    NUMBER = {3-4},
     PAGES = {235--265},
      ISSN = {0747-7171,1095-855X},
   MRCLASS = {68Q40}
}
%
\bib{BC18}{article}{
    AUTHOR = {Bauer, I.},
    AUTHOR = {Catanese, F.},
     TITLE = {On rigid compact complex surfaces and manifolds},
   JOURNAL = {Adv. Math.},
  FJOURNAL = {Advances in Mathematics},
    VOLUME = {333},
      YEAR = {2018},
     PAGES = {620--669},
      ISSN = {0001-8708,1090-2082},
   MRCLASS = {14D06 (14B12 14D07 14E20 14F17 14J15 14J29 32J15)}
}
%
\bib{BGV15}{proceedings}{
     TITLE = {Beauville surfaces and groups},
    SERIES = {Springer Proceedings in Mathematics \& Statistics},
    VOLUME = {123},
 BOOKTITLE = {Proceedings of the conference held at the {U}niversity of
              {N}ewcastle, {N}ewcastle-upon-{T}yne, {J}une 7--9, 2012},
    EDITOR = {Bauer, I.},
    EDITOR = {Garion, S.},
    EDITOR = {Vdovina, A.},
 PUBLISHER = {Springer, Cham},
      YEAR = {2015},
     PAGES = {ix+183},
      ISBN = {978-3-319-13862-6; 978-3-319-13861-9},
   MRCLASS = {14-06}
}
%
%\bib{BGK22}{article}{
	%	title={On Rigid Manifolds of Kodaira Dimension 1}, 
	%	author = {Bauer, I.},
     %   author = {Gleissner, C.},
      %  author = {Kotonski, J.},
		%year={2022},
		%eprint={2212.05859},
		%archivePrefix={arXiv},
%}
%
\bib{BGK25}{article}{
title={On Rigid Manifolds of Kodaira Dimension 1},
author = {Bauer, I.},
author = {Gleissner, C.},
author = {Kotonski, J.},
bookTitle={Perspectives on Four Decades of Algebraic Geometry, Volume 1: In Memory of Alberto Collino},
year={2025},
publisher={Springer Nature Switzerland},
pages={43--71}
}
%
\bib{BG20}{article}{
    AUTHOR = {Bauer, I.},
    AUTHOR = {Gleissner, C.},
     TITLE = {Fermat's cubic, {K}lein's quartic and rigid complex manifolds
              of {K}odaira dimension one},
   JOURNAL = {Doc. Math.},
  FJOURNAL = {Documenta Mathematica},
    VOLUME = {25},
      YEAR = {2020},
     PAGES = {1241--1262},
      ISSN = {1431-0635,1431-0643},
   MRCLASS = {14B12 (14D06 14M25 32G05 32G07)},
  MRNUMBER = {4164723},
MRREVIEWER = {Enrica\ Floris},
}
%
\bib{BG21}{article}{
    author = {Bauer, I.},
	author = {Gleissner, C.},
	title = {Towards a Classification of Rigid Product Quotient Varieties of Kodaira Dimension $0$},
	journal = {Bollettino dell'Unione Matematica Italiana},
	publisher = {Springer-Verlag},
	year = {2021},
}
%
\bib{B82}{article}{
   author={Beauville, A.},
   title={Some remarks on K\"ahler manifolds with $c_{1}=0$},
   conference={
      title={Classification of algebraic and analytic manifolds},
      address={Katata},
      date={1982},
   },
   book={
      series={Progr. Math.},
      volume={39},
      publisher={Birkh\"auser Boston, Boston, MA},
   },
   date={1983},
   pages={1--26}
}
%
\bib{B83}{book}{
    AUTHOR = {Beauville, A.},
     TITLE = {Complex algebraic surfaces},
    SERIES = {London Mathematical Society Lecture Note Series},
    VOLUME = {68},
      NOTE = {Translated from the French by R. Barlow, N. I. Shepherd-Barron
              and M. Reid},
 PUBLISHER = {Cambridge University Press, Cambridge},
      YEAR = {1983},
     PAGES = {iv+132},
      ISBN = {0-521-28815-0},
   MRCLASS = {14J10 (14-02)},
  MRNUMBER = {732439}
}
%
\bib{CF22}{article}{
    AUTHOR = {Carta, L.},
   AUTHOR = {Fairbairn, B. T.},
     TITLE = {Generalised {B}eauville groups},
   JOURNAL = {J. Group Theory},
  FJOURNAL = {Journal of Group Theory},
    VOLUME = {25},
      YEAR = {2022},
    NUMBER = {2},
     PAGES = {355--377},
      ISSN = {1433-5883,1435-4446},
   MRCLASS = {14L30 (14J29 20D99)}
}
%
\bib{Cat00}{article}{
    AUTHOR = {Catanese, F.},
     TITLE = {Fibred surfaces, varieties isogenous to a product and related
              moduli spaces},
   JOURNAL = {Amer. J. Math.},
  FJOURNAL = {American Journal of Mathematics},
    VOLUME = {122},
      YEAR = {2000},
    NUMBER = {1},
     PAGES = {1--44},
      ISSN = {0002-9327,1080-6377},
   MRCLASS = {14J10 (14D06 14J80)},
  MRNUMBER = {1737256},
MRREVIEWER = {Ulf\ Persson},
       URL =
              {http://muse.jhu.edu/journals/american_journal_of_mathematics/v122/122.1catanese.pdf},
}
%
\bib{Cat15}{article}{
 Author = {Catanese, F.},
 Title = {Topological methods in moduli theory},
 FJournal = {Bulletin of Mathematical Sciences},
 Journal = {Bull. Math. Sci.},
 ISSN = {1664-3607},
 Volume = {5},
 Number = {3},
 Pages = {287--449},
 Year = {2015},
 Language = {English},
 DOI = {10.1007/s13373-015-0070-1},
 Keywords = {14J29,14D20,14J10,14J15}
}
%
\bib{CGP23}{article}{
    AUTHOR = {Conti, D.},
    AUTHOR = {Ghigi, A.},
    AUTHOR = {Pignatelli, R.},
     TITLE = {Topological types of actions on curves},
   JOURNAL = {J. Symbolic Comput.},
  FJOURNAL = {Journal of Symbolic Computation},
    VOLUME = {118},
      YEAR = {2023},
     PAGES = {17--31},
      ISSN = {0747-7171,1095-855X},
   MRCLASS = {14Q05 (14H15 14H37 14J10 57M60)}
}
%
\bib{DG23}{article}{
    AUTHOR = {Demleitner, A.},
    AUTHOR = {Gleissner, C.},
     TITLE = {The classification of rigid hyperelliptic fourfolds},
   JOURNAL = {Ann. Mat. Pura Appl. (4)},
  FJOURNAL = {Annali di Matematica Pura ed Applicata. Series IV},
    VOLUME = {202},
      YEAR = {2023},
    NUMBER = {3},
     PAGES = {1425--1450},
      ISSN = {0373-3114,1618-1891},
   MRCLASS = {14J35 (14J10 14L30 20C15 20H15 32G05 32Q15)}
}
%
\bib{DP12}{article}{
    AUTHOR = {Dedieu, T.},
    AUTHOR = {Perroni, F.},
     TITLE = {The fundamental group of a quotient of a product of curves},
   JOURNAL = {J. Group Theory},
  FJOURNAL = {Journal of Group Theory},
    VOLUME = {15},
      YEAR = {2012},
    NUMBER = {3},
     PAGES = {439--453},
      ISSN = {1433-5883,1435-4446},
   MRCLASS = {14F25}
}
%
\bib{CP09}{article}{
    AUTHOR = {Carnovale, G.},
     AUTHOR = {Polizzi, F.},
     TITLE = {The classification of surfaces with {$p_g=q=1$} isogenous to a
              product of curves},
   JOURNAL = {Adv. Geom.},
  FJOURNAL = {Advances in Geometry},
    VOLUME = {9},
      YEAR = {2009},
    NUMBER = {2},
     PAGES = {233--256},
      ISSN = {1615-715X,1615-7168},
   MRCLASS = {14J29 (14J10)}
}
%
\bib{fede24}{article}{
		title={On the classification of product-quotient surfaces with 
                 $q = 0$, $p_g = 3$ and their canonical map}, 
		author = {Fallucca, F.},
		year={2024},
		eprint={arXiv:2405.04425},
		archivePrefix={arXiv},
        JOURNAL={to appear in Rend. Lincei (9) Mat. Appl.}
}
%
\bib{FG16}{article}{
    AUTHOR = {Frapporti, D.}
    AUTHOR = {Glei\ss ner, C.},
     TITLE = {On threefolds isogenous to a product of curves},
   JOURNAL = {J. Algebra},
  FJOURNAL = {Journal of Algebra},
    VOLUME = {465},
      YEAR = {2016},
     PAGES = {170--189},
      ISSN = {0021-8693,1090-266X},
   MRCLASS = {14J30 (14L30 14Q15)}
}
%
\bib{G15}{incollection}{
    AUTHOR = {Glei\ss ner, C.},
     TITLE = {The classification of regular surfaces isogenous to a product
              of curves with {$\chi(\mathcal O_S)=2$}},
 BOOKTITLE = {Beauville surfaces and groups},
    SERIES = {Springer Proc. Math. Stat.},
    VOLUME = {123},
     PAGES = {79--95},
 PUBLISHER = {Springer, Cham},
      YEAR = {2015},
      ISBN = {978-3-319-13862-6; 978-3-319-13861-9},
   MRCLASS = {14L30 (14J10)}
}


\bib{N71}{book}{
    AUTHOR = {Narasimhan, R.},
     TITLE = {Several complex variables},
    SERIES = {Chicago Lectures in Mathematics},
      NOTE = {Reprint of the 1971 original},
 PUBLISHER = {University of Chicago Press, Chicago, IL},
      YEAR = {1995},
     PAGES = {x+174},
      ISBN = {0-226-56817-2},
   MRCLASS = {32-01}
}
\bib{Pa23}{article}{
    AUTHOR = {Paulhus, J.},
     TITLE = {A database of group actions on {R}iemann surfaces},
 BOOKTITLE = {Birational geometry, {K}\"ahler-{E}instein metrics and
              degenerations},
    SERIES = {Springer Proc. Math. Stat.},
    VOLUME = {409},
     PAGES = {693--708},
 PUBLISHER = {Springer, Cham},
      YEAR = {2023},
      ISBN = {978-3-031-17858-0; 978-3-031-17859-7},
   MRCLASS = {30F35 (14H37 20H10 30-04)}
}
\bib{P11}{article}{
    AUTHOR = {Penegini, M.},
     TITLE = {The classification of isotrivially fibred surfaces with
              {$p_g=q=2$}},
      NOTE = {With an appendix by S\"onke Rollenske},
   JOURNAL = {Collect. Math.},
  FJOURNAL = {Collectanea Mathematica},
    VOLUME = {62},
      YEAR = {2011},
    NUMBER = {3},
     PAGES = {239--274},
      ISSN = {0010-0757,2038-4815},
   MRCLASS = {14J29 (14D06)}
}
\end{thebibliography}
\end{document}